\documentclass[a4paper, reqno]{amsart}

\usepackage[utf8]{inputenc}
\usepackage[usenames,dvipsnames]{color}
\usepackage{amsthm,amsfonts,amssymb,amsmath,amsxtra}
\usepackage[all]{xy}
\SelectTips{cm}{}
\usepackage{xr-hyper}
\usepackage[colorlinks=true,
citecolor=black,
linkcolor=red,
urlcolor=blue]{hyperref}
\usepackage{verbatim}
\usepackage{caption}

\usepackage[margin=1.25in]{geometry}
\usepackage{mathrsfs}

\usepackage[symbol]{footmisc}

\usepackage{tabularx}
\usepackage{makecell}

\usepackage{booktabs}

\usepackage{tikz}

\makeatletter
\newcommand*\circled[2][1.6]{\tikz[baseline=(char.base)]{
    \node[shape=circle, draw, inner sep=1pt,
        minimum height={\f@size*#1},] (char) {\vphantom{WAH1g}#2};}}
\makeatother

\RequirePackage{xspace}
\RequirePackage{etoolbox}
\RequirePackage{varwidth}
\RequirePackage{enumitem}
\RequirePackage{tensor}
\RequirePackage{mathtools}
\RequirePackage{longtable}
\RequirePackage{multirow}

\def\Cox{\mathop{\rm Cox}}
\def\KCox{{}^K\mathop{\rm Cox}}
\def\ge{\geqslant}
\def\le{\leqslant}
\def\a{\alpha}

\def\o{\omega}

\def\s{\sigma}
\def\t{\tau}
\def\th{\theta}

\def\l{\lambda}

\def\i{^{-1}}

\def\<{\langle}
\def\>{\rangle}

\newcommand{\bG}{\mathbf G}

\newcommand{\bJ}{\mathbf J}

\newcommand{{\BG}}{\ensuremath{\mathbb {G}}\xspace}

\newcommand{{\BK}}{\ensuremath{\mathbb {K}}\xspace}

\newcommand{\BR}{\ensuremath{\mathbb {R}}\xspace}
\newcommand{\BS}{\ensuremath{\mathbb {S}}\xspace}

\newcommand{\BZ}{\ensuremath{\mathbb {Z}}\xspace}

\newcommand{\ssrk}{{\mathrm{rank}_F^{\mathrm{ss}}}}

\newcommand{\Ad}{{\mathrm{Ad}}}

\DeclareMathOperator{\Adm}{Adm}
\newcommand{\KAdm}{{}^{K}\!{\rm Adm}}

\newcommand{\GL}{\mathrm{GL}}

\newcommand{\id}{\ensuremath{\mathrm{id}}\xspace}

\DeclareMathOperator{\Res}{Res}

\def\tW{\tilde W}
\def\tS{\tilde \BS}
\def\kk{\mathbf k}
\DeclareMathOperator{\supp}{supp}
\newcommand\brI{\breve{\mathcal I}}
\newcommand\brK{\breve{\mathcal K}}

\newcommand\brG{\breve{G}}
\newcommand\brQp{\breve{\mathbb Q}_p}

\newcommand{\Pcirc}{P^\circ}
\newcommand{\Pstab}{P}
\newcommand{\Pprimecirc}{{P'}^\circ}
\newcommand{\Pprimestab}{P'}
\newcommand{\Gr}{{\rm Gr}}

\newcommand\newresult{\circled[1.2]{\textsf{\textbf{N}}}\ }
\newcommand\newresultA{\circled[0.7]{\textsf{\textbf{N}${}^*$}}\ }
\newcommand\opencase{\circled[1.2]{\textsf{\textbf{?}}}\ }
\newcommand\knownresult{\circled[1.2]{\checkmark}\ }
\newtheorem{theorem}{Theorem}
\newtheorem{proposition}[theorem]{Proposition}
\newtheorem{lemma}[theorem]{Lemma}

\newtheorem{corollary}[theorem]{Corollary}

\theoremstyle{definition}
\newtheorem{definition}[theorem]{Definition}
\newtheorem{example}[theorem]{Example}

\newtheorem{remark}[theorem]{Remark}

\numberwithin{equation}{section}
\numberwithin{theorem}{section}

\setitemize[0]{leftmargin=*,itemsep=\the\smallskipamount}
\setenumerate[0]{leftmargin=*,itemsep=\the\smallskipamount}

\renewcommand{\to}{%
   \ifbool{@display}{\longrightarrow}{\rightarrow}%
   }
\let\shortmapsto\mapsto
\renewcommand{\mapsto}{%
   \ifbool{@display}{\longmapsto}{\shortmapsto}%
   }
\newlength{\olen}
\newlength{\ulen}
\newlength{\xlen}
\newcommand{\xra}[2][]{%
   \ifbool{@display}%
      {\settowidth{\olen}{$\overset{#2}{\longrightarrow}$}%
       \settowidth{\ulen}{$\underset{#1}{\longrightarrow}$}%
       \settowidth{\xlen}{$\xrightarrow[#1]{#2}$}%
       \ifdimgreater{\olen}{\xlen}%
          {\underset{#1}{\overset{#2}{\longrightarrow}}}%
          {\ifdimgreater{\ulen}{\xlen}%
             {\underset{#1}{\overset{#2}{\longrightarrow}}}
             {\xrightarrow[#1]{#2}}}}%
      {\xrightarrow[#1]{#2}}
   }
\makeatother
\newcommand{\xyra}[2][]{%
   \settowidth{\xlen}{$\xrightarrow[#1]{#2}$}%
   \ifbool{@display}%
      {\settowidth{\olen}{$\overset{#2}{\longrightarrow}$}%
       \settowidth{\ulen}{$\underset{#1}{\longrightarrow}$}%
       \ifdimgreater{\olen}{\xlen}%
          {\mathrel{\xymatrix@M=.12ex@C=3.2ex{\ar[r]^-{#2}_-{#1} &}}}%
          {\ifdimgreater{\ulen}{\xlen}%
             {\mathrel{\xymatrix@M=.12ex@C=3.2ex{\ar[r]^-{#2}_-{#1} &}}}
             {\mathrel{\xymatrix@M=.12ex@C=\the\xlen{\ar[r]^-{#2}_-{#1} &}}}}}%
      {\mathrel{\xymatrix@M=.12ex@C=\the\xlen{\ar[r]^-{#2}_-{#1} &}}}%
   }
\makeatletter
\newcommand{\xla}[2][]{%
   \ifbool{@display}%
      {\settowidth{\olen}{$\overset{#2}{\longleftarrow}$}%
       \settowidth{\ulen}{$\underset{#1}{\longleftarrow}$}%
       \settowidth{\xlen}{$\xleftarrow[#1]{#2}$}%
       \ifdimgreater{\olen}{\xlen}%
          {\underset{#1}{\overset{#2}{\longleftarrow}}}%
          {\ifdimgreater{\ulen}{\xlen}%
             {\underset{#1}{\overset{#2}{\longleftarrow}}}
             {\xleftarrow[#1]{#2}}}}%
      {\xleftarrow[#1]{#2}}
   }
\newcommand{\isoarrow}{%
   \ifbool{@display}{\overset{\sim}{\longrightarrow}}{\xrightarrow\sim}%
   }

\begin{document}

\title[Basic loci]{Basic loci of Coxeter type\\with arbitrary parahoric level}

\author{Ulrich G\"ortz}
\address{U.~G., University of Duisburg-Essen, Fakult\"at f\"ur Mathematik, 45117 Essen, Germany}
\email{ulrich.goertz@uni-due.de}
\author{Xuhua He}
\address{X.~H., The Institute of Mathematical Sciences and Department of Mathematics, The Chinese University of Hong Kong, Shatin, N.T., Hong Kong SAR, China}
\email{xuhuahe@math.cuhk.edu.hk}
\author{Sian Nie}
\address{S.~N., Institute of Mathematics, Academy of Mathematics and Systems Science, Chinese Academy of Sciences, 100190, Beijing, China}
\email{niesian@amss.ac.cn}

\thanks{}



\begin{abstract}
    Motivated by the desire to understand the geometry of the basic loci in the reduction of Shimura varieties, we study their ``group-theoretic models'' --- generalized affine Deligne--Lusztig varieties --- in cases where they have a particularly nice description. Continuing the work of~\cite{GH} and~\cite{GHN} we single out the class of cases \emph{of Coxeter type}, give a characterization in terms of the dimension, and obtain a complete classification. We also discuss known, new and open cases from the point of view of Shimura varieties/Rapoport--Zink spaces.
\end{abstract}

\maketitle

\tableofcontents

\section{Introduction}

\subsection{}
One way to understand arithmetic information encoded in Shimura varieties is to study the geometry and cohomology of the special fiber of a suitable integral model. This method has been applied with great success for example in the work of Harris and Taylor on the local Langlands correspondence for $GL_n$. The moduli description which is available at least in the case of Shimura varieties of Hodge type, together with the particular structure obtained (ultimately) from the Frobenius morphism, allows to define certain stratifications of the special fiber whose strata can be studied one by one. In particular, there is the Newton stratification which is defined, roughly speaking, by grouping those points into one stratum whose corresponding abelian varieties have isogenous $p$-divisible groups. Another important stratification is the EKOR stratification defined in~\cite{HR:axioms}, a stratification which simultaneously generalizes the Ekedahl--Oort stratification for hyperspecial level and the Kottwitz--Rapoport stratification for Iwahori level.

It has been observed, over the past few decades, that in certain cases the unique closed Newton stratum, the so-called \emph{basic locus}, has a simple description. More precisely, it has a stratification as a union of Deligne--Lusztig varieties, where the index set of the union and the combinatorics of the closure relations can be described in terms of a certain Bruhat--Tits building attached to the situation at hand. While for the Siegel moduli space of principally polarized $g$-dimensional abelian varieties this works only when $g \le 2$, there are several infinite families where such a description is possible; the cases studied in most detail so far arise from unitary Shimura varieties for unitary groups of signature $(1, n-1)$. See the paper~\cite{Vollaard-Wedhorn} by Vollaard and Wedhorn for a prototypical example, and Section~\ref{sec:individual-cases} for a more detailed discussion of individual cases and further references. Results of this type have found applications of different kinds:
\begin{itemize}
\item Explicit descriptions of the basic locus have been used to compute intersection numbers of special cycles on the special fiber of a Shimura variety, in order to prove results predicted by the Kudla--Rapoport program which relates such intersection numbers to Fourier coefficients of modular forms (in a general sense). As examples, we mention~\cite{Kudla-Rapoport-2}, \cite{Terstiege}, \cite{Li-Zhu}.

\item Similar intersection numbers on Rapoport--Zink spaces play a role in the arithmetic fundamental lemma and in arithmetic transfer conjectures.
See for instance \cite{Zhang:AFL}, \cite{Rapoport-Smithling-Zhang}.

\item In a different direction, a good understanding of the basic locus has been of high importance for some recent results around the Tate conjecture for the special fiber of certain Shimura varieties. See for example \cite{Helm-Tian-Xiao}, \cite{Tian-Xiao-2}, \cite{Xiao-Zhu}.
\end{itemize}

\smallskip

In this paper, we give a group-theoretic view on this phenomenon, extending previous work (see~\cite{GH},~\cite{GHN}) in this direction.

\subsection{}
Let us explain the main results of this paper.
We fix a connected reductive group $\bG$ over a non-archimedean local field $F$ and a conjugacy class of cocharacters $\mu$ of $\bG$ over a (fixed) algebraic closure of $F$. Let $\t\in B(\bG, \mu)$ be the unique basic element. Fix a rational level structure $K$. See Section~\ref{sec:preliminaries} for the notation used here and for more details.

The central object of this paper is the generalized \emph{affine Deligne--Lusztig variety} $X(\mu, \t)_K$, which can be viewed as a group-theoretic model of the basic locus mentioned above, in those cases where $\bG$ and $\mu$ come from a Shimura datum. This is a perfect scheme, locally perfectly of finite type over an algebraic closure of the residue class field of $F$, when $F$ has mixed characteristic. It is a scheme locally of finite type over an algebraic closure of the residue class field of $F$, if $F$ has equal characteristic.

The following definition (which originates from~\cite{GH}) singles out a class of particularly well-behaved cases. The idea behind it is to express the condition that $X(\mu, \t)_K$ is a union of (perfections of) classical Deligne--Lusztig varieties attached to a twisted Coxeter element (in some finite Weyl group).

We define (cf.~Definition~\ref{def:coxeter-type} for further details and equivalent formulations)

\begin{definition}
    The datum $(\bG, \mu, K)$ is said to be \emph{of Coxeter type} if every EKOR stratum that occurs in $X(\mu, \t)_K$ is the EKOR stratum of a Weyl group element $w$ that is a twisted Coxeter element in a finite standard parabolic subgroup of the Iwahori-Weyl group $\tW$.
\end{definition}

The notion of EKOR strata that we use here is the local version of the EKOR strata introduced in \cite{HR:axioms}, an interpolation between Ekedahl--Oort and Kottwitz--Rapoport strata. See Sections~\ref{def-fullyHN} and~\ref{subsec:bt-stratification} for further details.

The main novelty in this paper are new characterizations of the cases of Coxeter type, on the one hand by a simple dimension condition, on the other hand, equivalently, by an explicit group-theoretic condition which involves neither affine Deligne--Lusztig varieties nor the $\mu$-admissible set. As a consequence, we obtain a classification of all Coxeter cases.

We start by establishing the following general lower bound
\begin{theorem} {\rm (Theorem~\ref{prop:dim-lower-bound})}
    Let $\bJ_\t$ denote the $\s$-centralizer of $\t$.
	Suppose that $\mu$ is non-central in every simple factor of the adjoint group $\bG_{\text{ad}}$ over $F$. We have that
    \[
        \dim X(\mu, \t)_K \ge \ssrk(\bJ_\t).
    \]
\end{theorem}

We can characterize the cases that are of Coxeter type as precisely those cases where equality holds in the previous theorem:

\begin{theorem} {\rm (Theorem~\ref{thm:characterization-coxeter-type})}
Suppose that $\mu$ is non-central in every simple factor of the adjoint group $\bG_{\text{ad}}$ over $F$. The following conditions are equivalent:
	\begin{enumerate}
		\item The triple $(\bG, \mu, K)$ is of Coxeter type;
        \item We have that $\dim X(\mu, \t)_K=\ssrk(\bJ_\t)$;
	    \item For any admissible triple $(\xi, J, K')$ with $K' \supseteq K$, we have that
            \[
                \<\underline{\mu}, 2\rho\> \le \sharp \{\text{$\s$-orbits of $K_\xi'$}\} + \ssrk(\bJ_\t).
            \]
	\end{enumerate}
\end{theorem}

We list Condition~(3) in this theorem to indicate that we have a simple group-theoretic characterization of these cases which involves neither the dimension of $X(\mu, \t)_K$, nor the more subtle combinatorics of the admissible set. For the notation used here, we refer to Section~\ref{sec:classification}. This condition allows us to classify all cases of Coxeter type (without using the classification results of~\cite{GH}, \cite{GHN}).

\begin{theorem}[Theorem~\ref{thm:characterization-coxeter-type}, Table~\ref{table:enhanced-coxeter-data}]\label{thm:classification-coxeter-type}
    Assume that $\bG$ is quasi-simple over $F$ and $\mu$ is non-central in every $\breve F$-simple component. Denote by
    $W_a$ the corresponding affine Weyl group, by $\underline{\mu}$ the image of a dominant representative of $\mu$ in the translation lattice of the Iwahori--Weyl group, and by $\s$ the automorphism of
    $W_a$ induced by Frobenius.

    The property whether $(\bG, \mu, K)$ is of Coxeter type depends only on the tuple $(W_a, \s, \underline{\mu}, K)$.

    The quadruples $(W_a, \s, \underline{\mu}, K)$ of Coxeter type with $K$ minimal
    are classified as follows (up to isomorphism, see Section~\ref{notation-autom-dynkin} for the notation):

    \bigskip
    \begin{enumerate}[itemsep=2mm]
        \item[(i)]
            $(\tilde A_{n-1}, \id, \o^\vee_1, \emptyset)$,
        \item[(ii)]
            $(\tilde A_{n-1}, \varrho_{n-1}, \o^\vee_1, \emptyset)$,
        \item[(iii)]
            $(\tilde A_{2 m}, \varsigma_0, \o^\vee_1, \tilde \BS-\{s_0\})$,\\[2mm]
            $(\tilde A_{2 m+1}, \varsigma_0, \o^\vee_1, \tilde \BS-\{s_0, s_{m+1}\})$,
        \item[(iv)]
            $(\tilde A_{n-1}, \id, \o^\vee_1+\o^\vee_{n-1}, \tilde \BS-\{s_0\})$, \quad $n \ge 3$,\\[2mm]
            $(\tilde A_{n-1} \times \tilde A_{n-1}, {}^1 \varsigma_0, (\o_1^\vee, \o_{n-1}^\vee), \sqcup_{i=1}^2(\tS_i-\{0\}))$,
        \item[(v)]
            $(\tilde B_n, \id, \o^\vee_1, \tilde \BS-\{s_0, s_n\})$,\\[2mm]
            $(\tilde B_n, \Ad(\t_1), \o^\vee_1, \tilde \BS-\{s_n\})$,
        \item[(vi)]
            $(\tilde C_n, \id, \o^\vee_1, \tilde \BS-\{s_0, s_n\})$,
        \item[(vii)]
            $(\tilde D_n, \id, \o^\vee_1, \tilde \BS-\{s_0, s_n\})$,\\[2mm]
            $(\tilde D_n, \varsigma_0, \o^\vee_1, \tilde \BS-\{s_0\})$,
        \item[(viii)]
            Exceptional cases:\\[2mm]
            $(\tilde A_1, \id, 2 \o^\vee_1, \emptyset)$,\quad
            $(\tilde A_3, \id, \o^\vee_2, \{s_1, s_2\})$,\quad
            $(\tilde A_3, \varsigma_0, \o^\vee_2, \tilde \BS-\{s_0\})$,\\[2mm]
            $(\tilde C_2, \id, \o^\vee_2, \{s_0\})$,\quad
            $(\tilde C_2, \Ad(\t_2), \o^\vee_2, \{s_0, s_2\})$.
    \end{enumerate}
\end{theorem}

It is easy to see (Remark~\ref{KK'}) that whenever $(W_a, \s, \underline{\mu}, K)$ is of Coxeter type and $K \subseteq K'$, then $(W_a, \s, \underline{\mu}, K')$ is of Coxeter type. From the classification,  we also obtain the following, quite surprising result, which does not seem to follow directly from the characterization above: In all cases, there is a \emph{unique} (up to isomorphism) minimal set $K\subset\tS$ such that $(W_a, \s, \underline{\mu}, K)$ is of Coxeter type. Note however that the situation is quite subtle: Starting with a datum $(\bG, \mu)$, the corresponding statement is \emph{not true} (see the $C\text{-}BC_2$ cases in Table~\ref{table:known-new-open}) because isomorphisms of the Dynkin diagram might not be automorphisms of the \emph{oriented} local Dynkin diagram.

In Section~\ref{sec:rz-spaces}, we discuss consequences of our results for Rapoport--Zink spaces. For many of the pairs $(\bG, \mu)$ in our list, a stratification of the reduced special fiber by classical Deligne--Lusztig varieties has already been established. In most cases, these results deal with maximal parahoric level structure. The stratification is called the ``Bruhat--Tits stratification'' because its index set is related to a Bruhat--Tits building.

It is expected --- and known in many cases --- that one can identify the perfection of the special fiber with a generalized affine Deligne--Lusztig variety of the form $X(\mu, b)_K$ (see Section~\ref{sec:rz-spaces}, in particular Property ($\diamondsuit$)). This makes the connection with the group-theoretic results, and the Bruhat--Tits stratification on that side of the story, see Section~\ref{subsec:bt-stratification}.

Proposition~\ref{prop:bt-for-intersection} below allows us to establish a stratification of the special fiber for many non-maximal level structures in cases of Coxeter type, before passing to the perfection. Note that to prove this result, we need to know a priori that $X(\mu, \t)_K$ has a Bruhat--Tits stratification.


\subsection{Comparison with previous results.}

This article can be seen as a continuation of~\cite{GH} where a classification of all cases of Coxeter type was obtained under the following additional (and, a priori, quite restrictive) assumption: $K = \tS -\{v\}$ is the complement of a single element of $\tS$ (and, as above, is assumed to be preserved by $\sigma$, i.e., $v$ is a fix point of $\sigma$).

In this paper we remove this restriction on $K$, and obtain different more conceptual characterizations of the cases of Coxeter type.

In~\cite{GHN} we introduced and studied the notion of \emph{fully Hodge--Newton decomposable} pairs $(\bG, \mu)$ (see Section~\ref{def-fullyHN}), and gave a classification of those cases. We will see below (Corollary~\ref{cor:coxeter-is-fullyHN}) that $(\bG, \mu)$ is fully Hodge--Newton decomposable whenever it is of Coxeter type. Note however that the classification results in the paper at hand do not make use of the classification of fully Hodge--Newton decomposable cases in~\cite{GHN}. While in all fully Hodge--Newton decomposable cases, the space $X(\mu, b)_K$ (for basic $b$) has a stratification into classical Deligne--Lusztig varieties (called the \emph{weak} Bruhat--Tits stratification in Section~\ref{subsec:bt-stratification}), this stratification has additional nice properties in the cases of Coxeter type. For example, as we show in Section~\ref{sec:smoothness}, in the large majority of cases all closures of strata are smooth.

\subsection{Outline of the paper.} We recall some preliminary notions and explain the general setting in Section~\ref{sec:prelim}.
After recalling the method of Deligne--Lusztig reduction in Section~\ref{sec:dl-reduction}, we prove the dimension formula characterizing Coxeter type cases in Section~\ref{sec:dimension-formula}, and prove the classification in Section~\ref{sec:classification}. In Section~\ref{sec:rz-spaces} we discuss consequences for Rapoport--Zink spaces and relate our results to previous work on the side of Shimura varieties and Rapoport--Zink spaces. See Table~\ref{table:known-new-open} in Section~\ref{subsec:known-new-open} for a summary. Finally, in Section~\ref{sec:smoothness} we answer the question which closures of Bruhat--Tits strate are smooth.

\subsection{Acknowledgments.} We would like to thank George Pappas for answering questions on parahoric subgroups, and Chao Li for pointing us to a mistake in~\cite{GH}, Section~7, cf.~Section~\ref{sec:smoothness} below.

\section{Coxeter type}\label{sec:prelim}

\subsection{Setup}\label{sec:preliminaries}

Let $F$ be a non-archimedean local field, fix an algebraic closure $\overline{F}$, denote by $\breve F$ the completion of
its maximal unramified extension $F^{\rm un}\subset \overline{F}$, and by $\s$ the Frobenius
automorphism of $\breve F$ over $F$. We usually think of $F$ being of \emph{mixed
characteristic}, i.e., $F$ is a finite extension of $\mathbb Q_p$. Everything has an \emph{equal-characteristic} counterpart, however, where $F$ is of the form $\mathbb
F_q((t))$, the Laurent series field over a finite field $\mathbb F_q$. In
either case, we denote by $p$ the residue characteristic of $F$.

We fix a connected reductive group $\bG$ over $F$. Write $\brG = \bG(\breve F)$. Let
$\brI \subseteq \brG$ be a $\s$-invariant Iwahori subgroup, our \emph{standard
Iwahori subgroup}, and let $T$ be a maximal torus of $G$ such that the alcove
$\mathfrak a$ corresponding to $\brI$, the \emph{standard alcove} in the
Bruhat-Tits building of $\bG$ over $\breve F$, lies in the apartment attached to
$T_{\breve F}$. Attached to this data, we have the extended affine Weyl group
$\tW$ and the (relative) finite Weyl group $W_0$. We fix a special vertex of the
base alcove and obtain a splitting $\tW = X_* \rtimes W_0$, where
$X_* := X_*(T)_{\Gamma_0}$ denotes the coinvariants of the cocharacter lattice of $T$
with respect to $\Gamma_0 = \mathop{\rm Gal}(\overline{F}/F^{\rm un})$.
See~\cite{Tits:Corvallis}, \cite[\S2]{GHR}.

We denote by $\tS$ the set of simple affine reflections (defined by our base alcove) inside the affine Weyl group $W_a\subseteq \tW$. The Frobenius $\sigma$ acts on $\tS$ (since by assumption the base alcove is fixed by $\s$). Likewise, if $\tau\in \tW$ has length $0$, then it fixes the base alcove and thus acts by conjugation on $\tS$; we denote this action by $\Ad(\t)$. For an element $w = w' \tau \in W_a\tau$, the $\s$-support $\supp_\s(w)$ is the smallest subset of $\tS$ which is $\Ad(\t)\circ\s$-stable and contains all simple affine reflections that occur in a reduced expression for $w'$. The final condition can also be rephrased as $w'\in W_{\supp_\s(w)}$, where for a subset $K\subseteq \tS$ we write $W_K$ for the subgroup of $W_a$ generated by the elements of $K$. We denote by ${}^K \tW$ the set of minimal length representatives of the cosets in $W_K\backslash \tW$.

For $b\in \brG$, we denote by $\bJ_b$ the $\s$-centralizer of $b$, i.e.,
\[
    \bJ_b(F) = \{ g\in \brG;\ g^{-1}b\s(g) = b\}.
\]
If $b$ is understood, we just write $\bJ$ instead of $\bJ_b$.

Below we always work with the unique reduced root system $\Phi$ underlying the relative root system of $\bG$ over $\breve F$ (the \emph{\'echelonnage root system}).

\subsection{(Enhanced) Tits data and Coxeter data}

To specify the classification results below, two types of data typically arise. On the one hand, we will refer to affine Weyl groups together with an automorphism and a coweight; this kind of data we will call \emph{Coxeter data}. On the other hand, we will refer to algebraic groups over $F$ together with a conjugacy class of cocharacters; this is what we will call \emph{Tits data} below (using Tits's translation between isomorphism classes of such data with local Dynkin diagrams). In both cases, we often \emph{enhance} these data by including a ``level structure'', i.e., a subset $K\subset\tS$ with $W_K$ finite, of the set of simple affine reflections. In the group case, this gives rise to a standard parahoric subgroup. Our level structure will always be assumed to be rational, i.e., $K$ is fixed by the automorphism $\sigma$ (the automorphism induced by the Frobenius over $F$ in the group case).

We hope that no confusion will arise between the notions of \emph{Coxeter datum} and of \emph{being of Coxeter type}, the latter being a property that certain Coxeter data have, and others do not --- similarly as some elements of a \emph{Coxeter group} are \emph{Coxeter elements}.

\begin{definition}[{cf.~\cite{HPR}, \cite[\S2.6]{GHR}}]
    \begin{enumerate}
        \item
            A \emph{Coxeter datum} (over $F$) is a tuple $((W_a, \tS), \sigma, \lambda)$ consisting of an affine Coxeter system, a length-preserving automorphism $\sigma$ and a $W_0$-conjugacy class $\lambda$ in $X_*$, the coweight lattice. Here $W_0$ denotes the finite Weyl group of the given affine Coxeter system. An \emph{enhanced Coxeter datum} is a tuple $((W_a, \tS), \sigma, \lambda, K)$ whose first three entries constitute a Coxeter datum and where $K\subsetneq\tS$ is a subset with $\sigma(K)=K$. Below, we often just write $W_a$ instead of $(W_a, \tS)$, or we replace this item by the corresponding affine Dynkin type.
        \item
            A \emph{Tits datum} (over $F$) is a tuple $(\tilde{\Delta}, \sigma, \lambda)$ consisting of an absolute local Dynkin diagram (cf.~\cite{Tits:Corvallis}), a diagram automorphism and a $W_0$-conjugacy class $\lambda$ in the coweight lattice $X_*$. An \emph{enhanced Tits datum} is a tuple $(\tilde{\Delta}, \sigma, \lambda, K)$ whose first three entries constitute a Tits datum and where $K$ is a type of rational parahoric subgroups in the corresponding group.
    \end{enumerate}
\end{definition}

\subsubsection{Notation for automorphisms of Dynkin diagrams}\label{notation-autom-dynkin}

We use the same labeling of the Coxeter graph as in Bourbaki \cite[Plate I--X]{Bour}. As in~\cite{GHR}, we use the following notation for automorphisms of affine Dynkin diagrams. In case the fundamental coweight $\o^\vee_i$ is minuscule, we denote the corresponding length $0$ element $\tau(t^{\omega^\vee_i})$ by $\t_i$; conjugation by $\t_i$ is a length preserving automorphism of $\tW$ which we denote by $\Ad(\t_i)$. For type $A_n$, the automorphism $\Ad(\t_i)$ is the rotation of the affine Dynkin diagram by $i$ steps (i.e., $s_0$ is mapped to $s_i$, $s_1$ is mapped to $s_{i+1}$, and so on), and we also denote it by $\varrho_i$.  We write $\varsigma_0$ for the automorphism which fixes the vertex $0$, and is the unique nontrivial diagram automorphism of the finite Dynkin diagram, if $W_0$ is of type $A_n, D_n$ (with $n \ge 5$) or $E_6$. For type $D_4$, we also denote by $\varsigma_0$ the diagram automorphism which interchanges $\a_3$ and $\a_4$.

For the product $\tilde{A}_{n-1}\times \tilde{A}_{n-1}$, we denote by ${}^1\varsigma_0$ the automorphism which switches the two factors.

\subsection{Fully Hodge--Newton decomposable pairs $(\bG, \mu)$}\label{def-fullyHN}

We now fix a conjugacy class $\mu$ of cocharacters $\mathbf G_{m, \overline{F}}\to \bG_{\overline{F}}$ over the algebraic closure $\overline{F}$ of $F$. We denote by $\mu_+\in X_*(T)$ the dominant representative of this conjugacy class, and by $\underline{\mu}$ the image of $\mu_+$ in the coweight lattice $X_* = X_*(T)_{\Gamma_0}$, i.e., the translation lattice of the Iwahori--Weyl group. Cf.~\cite[\S2.2]{GHR}.

We also fix a (representative in $\brG$ of a) length $0$ element $\t\in\tW$ whose $\s$-conjugacy class is the unique basic element in $B(\bG, \mu)$.

Denote by $X_w(b)$ the affine Deligne--Lusztig variety for $w\in\tW$ and $b\in\brG$, a subvariety of the affine flag variety for $\bG$, $X_w(b) = \{g\brI\in\brG/\brI;\ g^{-1}b\s(g)\in \brI w\brI\}$.

Let $\pi = \pi_K\colon \brG/\brI \to \brG/\brK$ denote the projection from the affine flag variety to the partial affine flag variety of level $K$ ($\brK$ denotes the standard parahoric subgroup of type $K$). Recall that
\[
    \Adm(\mu) = \{ w\in\tW;\ w \le t^{x(\underline{\mu})}\text{\ for some $x\in W_0$} \}
\]
denotes the $\mu$-admissible set. By definition,
\[
    X(\mu, \t)_K = \{ g\brK \in \brG/\brK;\ g^{-1}\t\s(g) \in \brK\Adm(\mu)\brK\}.
\]

We write $\KAdm(\mu) = \Adm(\mu)\cap {}^K\tW$ for the subset of $\Adm(\mu)$ consisting of all elements which are of minimal length in their right $W_K$-coset. Below, we sometimes write $X_{K, w}(\t) = \pi(X_w(\t))$.

It is shown in~\cite{He:KRconj}, \cite{GH} that
\[
X(\mu, \t)_K = \bigsqcup_{w\in \KAdm(\mu)} \pi(X_w(\t)).
\]
The subsets $\pi(X_w(\t))$ are called the \emph{EKOR strata} of $X(\mu, \t)_K$.

Let us recall the following characterization of fully Hodge--Newton decomposable pairs $(\bG, \mu)$ (see~\cite[Def.~3.1]{GHN}).

\begin{theorem} {\rm (\cite[Thm.~B]{GHN})}\label{fully-HN-dec}
    The pair $(\bG, \mu)$ is \emph{fully Hodge--Newton decomposable}, if the following equivalent conditions are satisfied:
    \begin{enumerate}
        \item
            The coweight $\mu$ is minute {\rm (\cite[Def.~3.2]{GHN})}.
        \item
            For every $w \in \KAdm(\mu)$ with $X_w(\t) \ne\emptyset$, we
            have that $W_{\supp_\s(w)}$ is finite.
    \end{enumerate}
\end{theorem}

Note that this property, as shown by condition~(1), is independent of $K$, and depends only on the Coxeter datum $(W_a, \s, \underline{\mu})$. Set $$\KAdm(\mu)_0=\{w \in \Adm(\mu) \cap {}^K \tW; W_{\supp_\s(w)} \text{ is finite}\}.$$ Then we can rewrite condition (2) as
\[
X(\mu, \t)_K = \bigsqcup_{w\in \KAdm(\mu)_0} \pi(X_w(\t)).
\]

\begin{definition}
    We say an element $w \in W_a \t$ is a \emph{$\s$-Coxeter element} (or is
    a \emph{twisted Coxeter element}) if from each $\Ad(\t) \circ \s$-orbit on $\tilde
    \BS$ at most one simple reflection appears in some (or equivalently, any)
    reduced expression of $w \t \i$.
\end{definition}

We denote by $\KCox(\mu) \subset \KAdm(\mu)_0$ the subset of $\KAdm(\mu)_0$ consisting of $\s$-Coxeter elements.\footnote{This set was denoted $\mathop{\rm EO}^K_{\sigma, {\rm cox}}$ in~\cite{GH}.} If $K=\emptyset$, then we may simply omit the superscript and write $\Adm(\mu)_0$ and $\Cox(\mu)$. We define

\begin{definition}\label{def:coxeter-type}
    We say that the triple $(\bG, \mu, K)$ or the corresponding quadruple $(W_a, \mu, \s, K)$ is \emph{of Coxeter type} if $(W_a, \s, \mu)$ is fully Hodge-Newton decomposable and $\KCox(\mu)=\KAdm(\mu)_0$.
\end{definition}

Note that this is the same definition as given in the introduction. Furthermore,
this definition is equivalent to the one in~\cite[\S5.1]{GH}.
In fact, the definition in loc.~cit.~is easily seen to be equivalent to saying
that
\[
    X(\mu, b)_K = \bigsqcup_{w\in \KCox(\mu)} \pi(X_w(b)).
\]
The equivalence of the two definitions then follows from~(2) in Theorem~\ref{fully-HN-dec}.

\begin{remark}\label{KK'}
    Let $K \subset K'$ be proper $\s$-stable subsets of $\tilde \BS$. If $(W_a, \s, \mu, K)$ is of Coxeter type, then  $\KCox(\mu)=\KAdm(\mu)_0$ and
    \[
        {}^{K'}\Cox(\mu)=\KCox(\mu) \cap {}^{K'} \tW=\KAdm(\mu)_0 \cap {}^{K'} \tW={}^{K'}\!\Adm(\mu)_0.
    \]
    Hence $(W_a, \s, \mu, K')$ is of Coxeter type.
\end{remark}




\subsection{The Bruhat--Tits stratification}\label{subsec:bt-stratification}

Let $(\bG, \mu)$ be fully Hodge--Newton decomposable. As pointed out above, we have
\[
X(\mu, \t)_K = \bigsqcup_{w\in \KAdm(\mu)_0} \pi(X_w(\t)),
\]
where as before $\pi\colon \brG/\brI \to \brG/\brK$ denotes the projection from the full affine flag variety to the partial affine flag variety of type $K$.

For each $w\in \KAdm(\mu)_0$, we define the following standard parahoric subgroups. Let $\mathcal P^\flat_w$ be the standard parahoric subgroup generated by $\supp_\s(w)$.  Let $\mathcal P_w$ be the standard parahoric generated by $\supp_\s(w)$ and $I(K, w, \s)$. Here $I(K, w, \s)$ is the maximal $\Ad(w)\circ\s$-stable subset of $K$. Then~\cite[Prop.~5.7]{GHN} shows that $W_{\supp_\s(w) \cup I(K, w, \s)}$ is finite.
Finally, let $\brK\subset \brG$ be the standard parahoric subgroup of type $K$, and let $\mathcal Q_w := \mathcal P_w^\flat \cap \brK$ be the intersection. 

We then have
\[
    \pi(X_w(\tau)) = \bigsqcup_{j\in \bJ(F)/(\bJ(F)\cap \mathcal P_w)} jY(w),
\]
where
\[
    Y(w) = \{ g\in \mathcal P_w^\flat / \mathcal Q_w;\ g\i \t\s(g)\in \mathcal Q_w \cdot_\s (\brI w \brI) \}.
\]
We can identify $\mathcal P_w^\flat / \mathcal Q_w$ with a classical partial flag variety for the maximal reductive quotient of the special fiber of the parahoric group scheme attached to $\mathcal P_w^\flat$. The variety $Y(w)$ is a ``fine Deligne--Lusztig variety'' in there, i.e., the image of a classical Deligne--Lusztig variety in the corresponding full flag variety.

More precisely, in the mixed characteristic case, $Y(w)$ is the perfection of a classical Deligne--Lusztig variety; even though we sometimes drop the adjective perfect for notational convenience, we do not have an actual scheme structure on $X(\mu, \t)_K$ and therefore cannot talk about the strata as usual schemes, if $F$ has mixed characteristic.

If $w$ is a twisted Coxeter element, then we have the following simple description of the set $I(K, w, \s)$ which appears in the definition of $\mathcal P_w$.

\begin{lemma}\label{descr-IKws}
    Suppose that $w$ is a $\s$-Coxeter element in $W_{\supp_\s(w)}$. Then $I(K,w, \s)$ is the set of all $s\in \tS - \supp_\s(w)$ with the following two properties:
    \begin{enumerate}
        \item[(a)] $s$ commutes with every element of $\supp_\s(w)$, and
        \item[(b)] the $\Ad(\t)\circ\s$-orbit of $s$ is contained in $K$.
    \end{enumerate}
\end{lemma}

\begin{proof}
    It is shown in~\cite[Lemma~4.6.1]{GH} that every element $s\in I(K, w, \s)$ commutes with all elements of $\supp_\s(w)$, and that $I(K, w, \s)\cap \supp_\s(w) = \emptyset$. It remains to show that $I(K,w, \s)$ is $\Ad(\t)\circ \s$-stable, and that every element of $\tS - \supp_\s(w)$ which satisfies~(a) and~(b) lies in $I(K, w, \s)$.

    By definition, $I(K, w, \s)$ is $\Ad(w)\circ \s$-stable. It is also $\Ad(w\t^{-1})$-stable by property~(a). It follows that $I(K,w, \s)$ is $\Ad(\t)\circ \s$-stable.

    Now let $s\in \tS - \supp_\s(w)$ such that~(a) and~(b) hold. We need to show that $(\Ad(w)\circ \s)^i (s) \in K$ for all $i\ge 1$. But (a) ensures that $\Ad(w\t^{-1})^{-1}(s) = s$, so $\Ad(w)\circ\s(s) = \Ad(\t)\circ\s(s)$, and this is an element of $K$ by~(b). Since $\Ad(\t)\circ\s(s)$ again satisfies~(a) and~(b), we can apply induction, and the lemma follows.
\end{proof}

In terms of the Dynkin diagram we can express this as saying that $I(K, w, \s)$ is the union of those $\Ad(\t)\circ\s$-orbits in $K$ in which no element is connected to any vertex in $\supp_\s(w)$.

In particular, this implies that the projection $\pi$ restricts to an isomorphism from $\{ g\brI;\ g\in {\mathcal P}^\flat_w, g^{-1}\tau\s(g)\in \brI w\brI\}$ onto its image $Y(w)$. So $Y(w)$ is isomorphic to the classical Deligne--Lusztig variety attached to $w\t^{-1}$ in the (finite-dimensional) flag variety $\mathcal P^\flat_w/\brI$, for the Frobenius given by $\Ad(\t)\circ \s$.

See also~\cite[Cor.~4.6.2, Section~7.2]{GH}, cf.~also~\cite[Section~5.10]{GHN} for further details.

Putting together these stratifications for all the different $w$, we obtain the decomposition of $X(\mu, \t)_K$ as a union of classical Deligne--Lusztig varieties ``in a natural way''. We call this stratification the \emph{weak Bruhat--Tits stratification}.

The closure of the stratum $Y(w)$ (and likewise of $jY(w)$) in $X(\mu, \t)_K$ is isomorphic to (the perfection of) the closure of $Y(w)$ inside $\mathcal P_w^\flat / (\mathcal P_w^\flat \cap \brK)$. Even if $w$ is a twisted Coxeter element, then this closure is typically not isomorphic to the closure of $Y(w)$ inside $\mathcal P_w^\flat /\brI$.

The closure relations between strata are given as follows (cf.~\cite[Sections 3.3 and 7]{GH}). The closure of a stratum is a union of strata. The closure of a stratum $j Y(w)$ contains a stratum $j' Y(w')$ if and only if
\begin{enumerate}
    \item
$w' \le_{K, \s}w$ which means by definition that there exists $u\in W_K$ such that $u^{-1} w' \s(u)\le w$, and
\item
    $j'(\bJ(F)\cap \mathcal P_{w'}) \cap j(\bJ(F)\cap \mathcal P_w) \ne \emptyset$.
\end{enumerate}
We can express the second condition in an equivalent way in terms of the building, as follows (see~\cite[Erratum, Prop.~7.2.2]{GH}). We identify the set
\[
    \{ h\in \bJ(F)/\bJ(F)\cap \mathcal P_w;\ \breve{\kappa}(h) = \breve{\kappa}(j)  \}
\]
with the set of simplices of type $\mathcal P_w$ in the rational building of $\bJ$, and similarly for $w'$. Then~(2) above is equivalent to requiring that $\breve{\kappa}(j') = \breve{\kappa}(j)$ and that the simplices attached to $j$ and $j'$ via the above identification are contained in the closure of some alcove.

If moreover $(\bG, \mu, K)$ is of Coxeter type, then this stratification has further nice properties. 

\begin{proposition}
   Let $(\bG, \mu, K)$ be of Coxeter type.
   Let $w, w' \in \KAdm(\mu)_0$. The following are equivalent:
   \begin{enumerate}
        \item
            $w' \le w$ (where $\le$ denotes the Bruhat order),
        \item
            $w' \le_{K, \s} w$ (where $\le_{K, \s}$ is the partial order arising in the above description of the closure relations between strata),
         \item
            the inclusion $\supp(w'\t\i) \subseteq \supp(w\t\i)$ holds,
        \item
            the inclusion $\supp_\s(w') \subseteq \supp_\s(w)$ holds.
    \end{enumerate}
    In particular, for $w, w' \in \KCox(\mu)$ with $w \ne w'$ we have $\supp_\s(w') \neq \supp_\s(w)$.
\end{proposition}

Note that by definition, we automatically have $(1) \Rightarrow (2) \Rightarrow (4)$ and $(1) \Rightarrow (3) \Rightarrow (4)$. The nontrivial part is $(4) \Rightarrow (1)$, which follows by analyzing all Coxeter type cases in Section~\ref{sec:classification} and the explicit description of $\Cox(\mu)_K$. See Section~\ref{example:drinfeld} for some detailed discussion for the Drinfeld case.

We have the following consequence.

\begin{corollary}\label{ikw}
    Let $(\bG, \mu, K)$ be of Coxeter type.
    The set $\supp_\s(w) \cup I(K, w, \s)$ determines the element $w\in \Cox(\mu)$.
    Hence for  $w, w' \in \KCox(\mu)$ with $w \ne w'$ we have $\bJ(F)\cap \mathcal P_w \ne \bJ(F)\cap \mathcal P_{w'}$.
\end{corollary}

\begin{proof}
    By the proposition above, $w$ is determined by $\supp_\s(w)$. So it is enough
    to show that we can recover $\supp_\s(w)$ from the union $\supp_\s(w) \cup
    I(K, w, \s)$. Since $w\in {}^K\tW$, every connected component of
    $\supp_\s(w)$ must meet $\tS-K$. On the other hand, by definition we have
    $I(K, w, \s)\subseteq K$. Therefore the $\s$-support of $w$ consists exactly
    of those connected components of the union $\supp_\s(w) \cup I(K, w, \s)$
    which intersect $\tS-K$.
\end{proof}

By Corollary \ref{ikw}, if $(\bG, \mu, K)$ is of Coxeter type, then the index set $\bigsqcup_w \bJ(F)/(\bJ(F)\cap \mathcal P_w)$ of the weak BT stratification can be seen as a subset of the set of all
simplices in the Bruhat--Tits building of $\bJ$ over $F$ (up to fixing the connected component, i.e., the image under $\breve\kappa$). In view of these
particularly favorable properties, we call the resulting stratification the
\emph{Bruhat--Tits stratification of $X(\mu, \t)_K$}.

\section{Some dimension formulas}\label{sec:dimension-formula}
We first prove (in Section~\ref{sec:proof-supp}) the following inequality on the dimension of affine Deligne-Lusztig varieties. As before, $\t$ is a fixed representative of a length $0$ element in $\tW$ whose $\s$-conjugacy class is the basic element in $B(\bG, \mu)$.

\begin{proposition}\label{supp}
    Let $w \in W_a \o$ with $\o \in \Omega$ such that $X_w(\t) \neq \emptyset$. Then \begin{align*} \dim X_w(\t) \ge \sharp \{\text{\rm $(\Ad(\o) \circ \s)$-orbits on $\supp_\s(w)$}\}. \end{align*}
\end{proposition}

\subsection{Deligne-Lusztig reduction}\label{sec:dl-reduction}

We first recall the Deligne-Lusztig reduction method.

Let $x, x' \in \tW$ and $s \in \tS$. We write $x \xrightarrow{s}_\s x'$ if $x' = s x \s(s)$ and $\ell(x') \le \ell(x)$. We write $x \to_\s x'$ if there exists a sequence $x_0, x_1, \dots, x_r$ in $\tW$ and a sequence $s_1, s_2, \dots, s_r$ in $\tS$ such that $x = x_0 \xrightarrow{s_1}_\s x_1 \xrightarrow{s_2}_\s \cdots \xrightarrow{s_r}_\s x_r = x'$. We write $x \approx_\s x'$ if $x \to_\s x'$ and $x' \to_\s x$.

\begin{theorem}~\cite{HN}\label{min}
    For each $x \in \tW$ there exists an element $y \in \tW$ which is of minimal length inside its $\s$-conjugacy class such that $x \to_\s y$;
\end{theorem}

The following theorem, which is referred to as the reduction \'a la Deligne and Lusztig, is proved in \cite[proof of Theorem 1.6]{DL} (parts~(i) and~(ii)) and \cite[Theorem 4.8]{He}, see also \cite[Corollary 2.5.3]{GH10}.
\begin{theorem} \label{reduction}
    Let $b \in \breve G$. Let $x, x' \in \tW$ such that $x \xrightarrow{s}_\s x'$ for some $s \in \tS$.
   \begin{enumerate}
       \item[(i)]
    if $\ell(x) = \ell(x')$, then $\dim X_x(b) = \dim X_{x'}(b)$;
       \item[(ii)]
    if $\ell(x) > \ell(x')$, then $\dim X_x(b) = 1 + \max \{\dim X_{x'}(b), \dim X_{s x}(b)\}$;
       \item[(iii)]
    if $x$ is of minimal length in its $\s$-conjugacy class, then $X_x(b) \neq \emptyset$ if and only if $\dot x \in [b]$, in which case, $\dim X_x(b) = \ell(x) - \<\bar{\nu}_b, 2\rho\>$, where $\bar{\nu}_b$ denotes the Newton vector of $b$.
   \end{enumerate}
\end{theorem}

\subsection{Proof of Proposition~\ref{supp}}\label{sec:proof-supp}
We argue by induction on the length of $w$. If $w$ is of minimal length in its $\s$-conjugacy class, by Theorem \ref{reduction},  $\dim X_w(\t) = \ell(w)$ and the statement follows. Otherwise, by Theorem \ref{min}, there exist $u \in \tW$ and $s \in \tS$ such that $w \approx_\s u$ and $s u \s(s) < u$. Thus, $\supp_\s(u) = \supp_\s(w)$ and $$\supp_\s(u) - \{(\Ad(\o) \circ \s)^i(s); i \in \BZ\} \subseteq \supp_\s(s u \s(s)).$$ Moreover, $\dim X_w(\t) = \dim X_u(\t)$ and either $X_{s u \s(s)}(\t) \neq \emptyset$ or $X_{s u}(\t) \neq \emptyset$. Let us assume that the former case occurs; the proof in the other case is basically the same. By induction hypothesis,
\begin{align*} \dim X_w (\t) &= \dim X_u (\t) \ge 1 + \dim X_{s u \s(s)}(\t ) \\ &\ge 1 +  \sharp \{\text{$(\Ad(\o) \circ \s)$-orbits on $\supp_\s(s u \s(s))$}\} \\ &\ge 1 +  \sharp \{\text{$(\Ad(\o) \circ \s)$-orbits on $\supp_\s(u)$}\} - 1 \\ &= \sharp \{\text{$(\Ad(\o) \circ \s)$-orbits on $\supp_\s(u)$}\} \\ &= \sharp \{\text{$(\Ad(\o) \circ \s)$-orbits on $\supp_\s(w)$}\}. \qedhere
\end{align*}

\subsection{A general dimension bound}
For any reductive group $\mathbf H$ over $F$, we denote by $\ssrk(\mathbf H)$ the semi-simple $F$-rank of $\mathbf H$. By \cite[\S 1.9]{Ko06}, if our group $\bG$ is quasi-simple over $F$, then $\ssrk(\bJ_\t) = \sharp \{\text{$(\Ad(\t) \circ \s)$-orbits on $\tS$}\} - 1$.

\begin{corollary}\label{cor-supp}
Let $w \in \tW$ with $\supp_\s(w) = \tS$ and $X_w(\t) \neq \emptyset$.
Then $\dim X_w(\t)> \ssrk(\bJ_\t)$.
\end{corollary}

\begin{proof} Let $\o \in \Omega$ such that $w \in W_a \o$.
    As $X_w(\t) \neq \emptyset$, there exists $\epsilon \in \Omega$ such that $\epsilon\i \o \s(\epsilon) = \t$.
This implies that $\ssrk(\bJ_\o) = \ssrk(\bJ_\t)$.
The result thus follows from Proposition~\ref{supp}.
\end{proof}

\smallskip

Now we prove the main result of this section.

\begin{theorem}\label{prop:dim-lower-bound}
Suppose that $\mu$ is non-central in every simple factor of the adjoint group $\bG_{\text{ad}}$ over $F$. Then
$$\dim X(\mu, \t)_K \ge \ssrk(\bJ_\t).$$
If moreover the equality holds, then $(\bG, \mu)$ is fully Hodge-Newton decomposable.
\end{theorem}

\begin{proof}
    We may reduce to the case where $\bG$ is quasi-simple over $F$ and that it is semisimple of adjoint type. Under this assumption, we have
    \begin{equation}
    \bG_{\breve F} = \bG_1 \times \cdots \times \bG_r,
    \end{equation}
    where each $\bG_i$ is a simple reductive group over $\breve F$, and $\s(\bG_j) = \bG_{j+1}$ for all $j$. Here we set $\bG_{r+1} = \bG_1$. Write $\tilde \BS = \tilde \BS_1 \sqcup \cdots \sqcup \tilde\BS_r$ and $\underline{\mu} = (\underline{\mu}_1, \dots, \underline{\mu}_r)$ with respect to the decomposition above. For $J \subseteq \tilde \BS$ we set $J_j = J \cap \tilde \BS_j$. We may write $\rho$ as $\rho=\rho_1+\ldots+\rho_l$, where $\rho_i$ is the half sum of positive roots corresponding to the root system associated to $\tS_i$. We also have that $\Omega=\Omega_1 \times \ldots \times \Omega_l$. We write $\t = (\t_1, \ldots, \t_l)$, where $\t_i \in \Omega_i$.

    It is easy to see that $\ell(t^{\mu_1}) = \<\mu_1, 2 \rho_1\> \ge \sharp\tS_1 \ge \ssrk(\bJ_\t)$. Let $\xi=(\xi_1, \ldots, \xi_l) \in W_0 \cdot \underline{\mu}$ such that $t^{\xi} \in {}^K \tW$. We choose a reduced expression $t^{\xi}=\t s_{i_1} \cdots s_{i_k}$ of $t^{\xi}$. Let $w=\t s_{i_1} \cdots s_{i_m}$, where $m=\ssrk(\bJ_\t)$. Then $w \le t^{\xi}$ and $w \in {}^K \tW$. Hence $w \in \KAdm(\mu)$. Since $\ell(w)=m<\sharp \{\text{$(\Ad(\t) \circ \s)$-orbits on $\tS$}\}$, the Weyl group $W_{\supp_\s(w)}$ is finite and hence $\dim X_{K, w}(\t)=\dim X_w(\t)=\ell(w)$. Thus $\dim X(\mu, \t)_K \ge \ell(w)=m$.

Now we assume that $\dim X(\mu, \t)_K = \ssrk(\bJ_\t)$. Let $w \in \KAdm(\mu)$ such that $K \cdot_\s I w I \cap [\t] \neq \emptyset$, that is, $I w I \cap [\t] \neq \emptyset$ or in other words, $X_w(\t) \neq \emptyset$. Then $ \dim X_w(\t)=\dim X_{K, w}(\t) \le \dim X(\mu, \t)_K = \ssrk(\bJ_\t)$. By Corollary~\ref{cor-supp}, we have $\supp_\s(w) \subsetneq \tS$ and hence $K \cdot_\s I w I \subseteq [\t]$. Noticing that $$K \Adm(\mu) K = \sqcup_{w \in \KAdm(\mu)} K \cdot_\s I w I,$$ we deduce that
    \begin{equation}\label{eq:fhnd}
    K \Adm(\mu) K \cap [\t] = \bigsqcup_{w \in \KAdm(\mu), \sharp W_{\supp_\s(w)} < \infty} K \cdot_\s I w I.
    \end{equation}
    By Section~\ref{def-fullyHN}, $(\bG, \mu)$ is fully Hodge-Newton decomposable.
\end{proof}

\section{Classification}\label{sec:classification}

Note that the fully Hodge-Newton decomposable cases are classified in \cite{GHN}. By further studying these cases via a case-by-case analysis, one may get a classification of the Coxeter types. However, there is a more direct approach (without using the classification of the Hodge-Newton decomposable cases). This approach classifies the cases where $\dim X(\mu, \t)_K = \ssrk(\bJ_\t)$, and in particular, by analyzing all these cases, we show that this equality implies that $(\bG, \mu, K)$ is of Coxeter type and thus we obtain a classification of the Coxeter types. This is what we will do in this section.

We may assume that $\bG$ is quasi-simple over $F$, and that it is semisimple of adjoint type (cf.~\cite[Section~3.3]{GHN}).
Under this assumption, we have a decomposition
\begin{equation}\label{eqn:decomposition}
\bG_{\breve F} = \bG_1 \times \cdots \times \bG_r,
\end{equation}
as in the proof of Theorem~\ref{prop:dim-lower-bound}.
As before, here each $\bG_i$ is a simple reductive group over $\breve F$, and $\s(\bG_j) = \bG_{j+1}$ for all $j$.  We set $\bG_{r+1} = \bG_1$. Write $\tilde \BS = \tilde \BS_1 \sqcup \cdots \sqcup \tilde\BS_r$ and $\underline{\mu} = (\underline{\mu}_1, \dots, \underline{\mu}_r)$ with respect to the decomposition above. For $J \subseteq \tilde \BS$ we set $J_j = J \cap \tilde \BS_j$.

We assume further each factor $\underline{\mu}_j$ is non-central.

Recall that $\Phi$ denotes the unique reduced root system underlying the relative root system of $\bG$ over $\breve F$ (the \emph{\'echelonnage root system}).

\subsection{Admissible triples} Let $s \in \tS$. If $s \in W_0$, set $\a_s$ to be the simple root corresponding to $s$. Otherwise, set $\a_s = -\th$, where $\th \in \Phi^+$ is the highest root of the $j$-component of the decomposition~\eqref{eqn:decomposition}, where $s\in \tS_j$; then $s = t^{\th^\vee} s_\th$.
Let $J \subseteq \tS$ such that $W_J$ is finite. We denote by $\Phi_J$ the root system spanned by $\a_s$ for $s \in J$. 

Let $p: \tW \rtimes \<\s\> \to \GL({X_*} \otimes \BR)$ be the natural projection.
\begin{lemma} \label{K-min}
    For $\l \in X_*$ we have $t^\l \in {}^K \tW$ if and only if $\<\l, \a_s\> \ge 0$ for all $s \in K$. In particular, there exists $\xi \in W_0 \cdot \underline{\mu}$ such that $t^\xi \in {}^K \tW$.
\end{lemma}
\begin{proof}
The first statement follows immediately from the definitions. For the second one, notice that $\{\a_s; s \in K\}$ is the set of simple roots for $\Phi_K$ whose Weyl group is $p(W_K) \subseteq W_0$. Thus, each $p(W_K)$-orbit in $W_0 \cdot \underline{\mu}$ contains a unique cocharacter $\xi$ such that $\<\xi, \a_s\> \ge 0$ for $s \in K$, that is, $t^\xi \in {}^K \tW$ as desired.
\end{proof}

Let $\xi \in W_0 \cdot \underline{\mu}$ and let $J \subsetneq \tS$ be a maximal proper $\s$-stable subset. Let $\xi_J \in \BR \Phi_J^\vee$ be such that $\<\xi_J, \a\> = \<\xi, \a\>$ for $\a \in \Phi_J$. We denote by $\xi_J^\diamond$ the $p(\s)$-average of $\xi_J$.

\begin{definition}
    We say the triple $(\xi, J, K)$ with $K = \s(K) \subseteq J$ is \emph{admissible} if $t^\xi \in {}^K \tW$ and $\xi_J^\diamond \in \BR \Phi_K^\vee$.
\end{definition}

In this case, we define $$K_\xi = \cup_C \cup_{i \in \BZ} \s^i(C) \subseteq K,$$ where $C$ ranges over the connected components $C$ of $K$ on which the $p(\s)$-average $\xi^\diamond$ is nonzero. In other words, $K_\xi$ is the minimal $\s$-stable subset of $K$ such that $\xi_J^\diamond \in \BR \Phi_{K_\xi}^\vee$.

\begin{lemma}\label{admissible}
    Let $(\xi, J, K)$ be an admissible triple. Then there exists some $\s$-Coxeter element $c \in W_{K_\xi}$ such that \begin{align*}\ell(t^\xi c) = \ell(t^\xi) - \ell(c) = \<\underline{\mu}, 2\rho\> - \sharp \{\text{$\s$-orbits of $K_\xi$}\}. \end{align*} In particular, $t^\xi c \in {}^K \tW \cap \Adm(\lambda)$ and the Newton point of $t^\xi c$ is central.
\end{lemma}
\begin{proof}
The existence of $c$ such that $\ell(t^\xi c) = \ell(t^\xi) - \ell(c)$ and hence $t^\xi c \in {}^K \tW \cap \Adm(\lambda)$ follows exactly along the same lines as~\cite[Lemma 6.4, Proposition 6.7]{GHN}. It remains to show that the Newton point of $t^\xi c$ is central. By the proof of \cite[Lemma 6.4]{GHN}, it suffices to show that the $p(c\s)$-average $\nu$ of $\xi_J$ is zero. Write $\xi_J = v' + v''$ such that $v'' \in \BR\Phi_{K_\xi}^\vee$ and $v'$ is orthogonal to $\BR\Phi_{K_\xi}^\vee$. As $c$ is a $\s$-Coxeter element of $W_{K_\xi}$, we see that $p(c\s) - \id$ is invertible on $\BR\Phi_{K_\xi}^\vee$, which means that $\nu$ equals the $p(\s)$-average of $v'$. In particular, $\nu$ is orthogonal to $\BR\Phi_{K_\xi}^\vee$. On the other hand, $\nu - \xi_J^\diamond \in \BR\Phi_{K_\xi}^\vee$. By assumption, $\xi_J^\diamond \in \BR\Phi_{K_\xi}^\vee$, which means $\nu \in \BR\Phi_{K_\xi}^\vee$ and $\nu = 0$ as desired.
\end{proof}

\begin{lemma}\label{ineq-1}
    Suppose that $\dim X(\mu, \t)_K=\ssrk(\bJ_\t)$ and that $(\xi, J, K')$ is an admissible triple such that $K' \supseteq K$. Then \begin{align*} \tag{a} \<\underline{\mu}, 2\rho\> \le \sharp \{\text{$\s$-orbits of $K_\xi'$}\} + \ssrk(\bJ_\t).\end{align*}
\end{lemma}

\begin{proof}
    Let $c$ be as in Lemma \ref{admissible}. Then we need to show that $\ell(t^\xi c) \le \ssrk(\bJ_\t)$. But our assumption implies, by Theorem~\ref{prop:dim-lower-bound}, that $(\bG, \mu)$ is fully Hodge--Newton decomposable, and hence we have $\ell(t^\xi c) = \dim X_{t^\xi c}(\t)$ (note that this is an easy consequence of~\eqref{eq:fhnd} and does not require the use of classification results).

    Altogether we obtain
    \[
        \ell(t^\xi c) = \dim X_{t^\xi c}(\t) \le \dim X(\mu, \t)_K=\ssrk(\bJ_\t),
    \]
    as desired.
\end{proof}

Given a $\s$-stable subset $K \subseteq \tS$ with $W_K$ finite, it follows from Lemma \ref{K-min} that there always exists some admissible triple $(\xi, J, K)$.

\begin{corollary}\label{ineq-2}
    If $\dim X(\mu, \t)_K=\ssrk(\bJ_\t)$, 
    then
    \begin{align*}
        \tag{b} \<\underline{\mu}, 2\rho\> \le \ssrk(\bG) + \ssrk(\bJ_\t).
    \end{align*}
    In particular,
    \[
        \tag{c} \<\underline{\mu}, 2\rho\>  \le 2\,{\rm rank}^{\rm ss}_{\overline{F}}(\bG).
\]
\end{corollary}

We can now state the following equivalent characterizations of being of Coxeter type. Note that there is an obvious notion of \emph{product of Coxeter data}. We call a Coxeter datum \emph{irreducible}, if it cannot be decomposed as a product in a non-trivial way.

\begin{theorem}\label{thm:characterization-coxeter-type}
    Consider an enhanced Tits datum  $(\bG, \mu, K)$ with corresponding enhanced Coxeter datum $(W_a, \s, \underline{\mu}, K)$. Assume that all components of $\underline{\mu}$ as in~\eqref{eqn:decomposition} are non-central. The following conditions are equivalent:
	\begin{enumerate}
		\item The enhanced Tits datum $(\bG, \mu, K)$ is of Coxeter type;
        \item We have that $\dim X(\mu, \t)_K=\ssrk(\bJ_\t)$;
	    \item For any admissible triple $(\xi, J, K')$ with $K' \supseteq K$, we have that
        \[
            \<\underline{\mu}, 2\rho\> \le \sharp \{\text{$\s$-orbits of $K_\xi'$}\} + \ssrk(\bJ_\t).
        \]
    \item The enhanced Coxeter datum $(W_a, \s, \underline{\mu}, K)$ is a product of irreducible enhanced Coxeter data, where for each factor the Coxeter datum is one of those listed in Table~\ref{table:enhanced-coxeter-data}, and the level structure $K$ contains the minimal one listed in that table. See Section~\ref{notation-autom-dynkin} for the notation.
	\end{enumerate}
\end{theorem}

From Theorem~\ref{prop:dim-lower-bound} we immediately get:

\begin{corollary}\label{cor:coxeter-is-fullyHN}
    If $(\bG, \mu, K)$ is of Coxeter type, then $(\bG, \mu)$ is fully Hodge--Newton decomposable.
\end{corollary}

\begin{table}[h!]
\renewcommand{\arraystretch}{1.7}
\begin{tabular}{@{}l@{\hskip1cm}l@{}}
\makecell{\textbf{Enhanced Coxeter datum}\\[.1cm]
$(\tW, \s, \underline{\mu}, K)$} & \makecell{\\[.1cm] $\KAdm(\mu)_0$} \\
\hline
$(\tilde A_{n-1}, \id, \o_1^\vee, \emptyset)$ & $\{\t\}$ \\
\hline
$(\tilde A_{n-1}, \varrho_{n-1}, \o_1^\vee, \emptyset)$ & $\Adm(\mu)$ \\
\hline
$(\tilde A_{2m}, \varsigma_0, \o_1^\vee, \tS-\{0\})$ & $\{s_{[2m+1, i]} \t;\ m+2 \le i \le 2m+2\}$ \\[.1cm]
\hline
$ (\tilde A_{2m+1}, \varsigma_0, \o_1^\vee, \tS - \{0, m+1\})$ & \makecell{ $\{s_{[2m+2, 2m+2-i]} s_{[m+1, m+1-j]} \t$;\\[.2cm] $i, j \ge -1$ with $i+j \le m-2\}$}\\
\hline
$(\tilde A_1, \id, 2 \o_1^\vee, \emptyset)$ & $\{1, s_0, s_1\}$ \\
\hline

$(\tilde A_{n-1}, \id, \o_1^\vee + \o_{n-1}^\vee, \tS-\{0\})$ for $n \ge 3$ & $\{1\} \sqcup \{s_{[n, i]} s_{[j+1, 1]}\i;\ 0 \le j+1 < i \le n\}$ \\
\hline
$(\tilde A_{n-1} \times \tilde A_{n-1}, {}^1 \varsigma_0, (\o_1^\vee, \o_{n-1}^\vee), \sqcup_{i=1}^2(\tS_i-\{0\}))$ & $\{(s_{[n, i]}, s_{[j, 0]}\i)\t;\ 0 \le j+1 < i \le n+1\}$ \\
\hline
$(\tilde A_3, \id, \o_2^\vee, \{1, 2\})$ & $\{\t, s_0 \t, s_3 \t\}$ \\
\hline
$(\tilde A_3, \varsigma_0, \o_2^\vee, \tS-\{0\})$ & $\{\t, s_0 \t, s_0 s_1 \t, s_0 s_3 \t\}$ \\
\hline
$(\tilde B_n, \id, \o_1^\vee, \tS-\{0, n\})$ & $\{\t s_{[i, 1]} \i s_{[n, j]};\ 0 \le i \le j-2 \le n-1\}$ \\[.1cm]
\hline
$(\tilde B_n, \Ad(\t_1), \o_1^\vee, \tS-\{n\})$ & \makecell{$\{\t, s_n \t, \dots, s_n s_{n-1} \cdots s_2 \t,$\\[.1cm] $s_n s_{n-1} \cdots s_2 s_1 \t, s_n s_{n-1} \cdots s_2 s_0 \t\}$}\\
\hline
$(\tilde C_n, \id, \o_1^\vee, \tS-\{0, n\})$ & $\{s_{[i, 0]} \i s_{[n, j]};\ -1 \le i \le j-2 \le n-1\}$ \\[.1cm]
\hline
$(\tilde C_2, \id, \o_2^\vee, \{0\})$ & $\{\t, s_1 \t, s_2 \t\}$ \\
\hline
$(\tilde C_2, \Ad(\t_2), \o_2^\vee, \{0, 2\})$ & $\{\t, s_1 \t, s_1 s_2 \t, s_1 s_0 \t\}$ \\
\hline
$(\tilde D_n, \id, \o_1^\vee, \tS-\{0, n\})$ & $\{\t s_{[i, 1]} \i s_{[n-1, j]};\ 0 \le i \le j-2 \le n-2\}$ \\[.1cm]
\hline
$(\tilde D_n, \varsigma_0, \o_1^\vee, \tS-\{0\})$ & \makecell{$\{\t, \t s_1, \cdots, \t s_1 s_2 \cdots s_{n-2}$,\\[.1cm] $\t s_1 s_2 \cdots s_{n-2} s_{n-1}, \t s_1 s_2 \cdots s_{n-2} s_n\}$} \\
\hline
\end{tabular}

\bigskip
\caption{The irreducible enhanced Coxeter data of Coxeter type (with the minimal level structure), up to isomorphism.\\
Notation: In type $\tilde A_{n-1}$, by convention we set $s_n=s_0$. We use the labeling of the affine Dynkin diagram as in~\cite{Bour}.
Set $s_{[a, b]}= s_a s_{a-1} \cdots s_b$ if $a \ge b$, and $s_{[a,b]} = 1$ otherwise.
}
\label{table:enhanced-coxeter-data}
\end{table}

\subsection{Strategy}
In Table~\ref{table:enhanced-coxeter-data}, we list the minimal irreducible enhanced Coxeter data (up to isomorphism) satisfying condition (3) of Theorem \ref{thm:characterization-coxeter-type} together with the set $\KAdm(\mu)_0$. It is easy to see that in all these cases, $\KAdm(\mu)_0=\KCox(\mu)$. Therefore, we also have ${}^{K'}\!\Adm(\mu)_0={}^{K'}\Cox(\mu)$ for all $K' \supset K$. This shows $(4) \Rightarrow (1)$. Note that $(1) \Rightarrow (2)$ is obvious and $(2) \Rightarrow (3)$ follows from Lemma \ref{ineq-1}. It remains to show that $(3) \Rightarrow (1)$.

Note that the condition (3), although a bit technical, is the most elementary one among the three conditions and only involves the root system. In the rest of the section, we will analyze the condition (3), and give a classification of the irreducible enhanced Coxeter data that satisfy this condition, and finally show that those cases are of Coxeter type. This finishes the direction $(3) \Rightarrow (1)$.

Our strategy is as follows.

In step (I), we show that the condition (3) implies that the Coxeter datum $(W_a, \s, \underline{\mu})$ is one of those listed in Table~\ref{tab:1}.
%
%
This is done by the inequalities (b) and (c) in Corollary \ref{ineq-2} in most of the cases. The only exception is in Type $D$, where in some cases we have to use the full strength of condition (3).

Note that the cases in this table are the fully Hodge-Newton decomposable cases. We will then further analyze these cases and give a complete classification.

In step (II), we show that the condition (3) implies that for the Coxeter datum in the Table above, the parahoric subgroup $K$ must be as specified in Theorem~\ref{thm:characterization-coxeter-type}/Table~\ref{table:enhanced-coxeter-data}.


\subsection{Step (I): The Coxeter datum $(\tW, \s, \underline{\mu})$}
Recall the decomposition~\eqref{eqn:decomposition}. We will argue on the type of the irreducible affine Dynkin diagram $\tilde \BS_j$, which does not depend on $j$. For $i \in \tilde\BS_j$ we denote by $\o_{i, j}^\vee$ the corresponding fundamental coweight in $\bG_j$. If $r = 1$, we write $\o_i^\vee = \o_{i, 1}^\vee$ for simplicity.

\subsubsection{Exceptional types} We use the inequality (c) to exclude all the exceptional types.

Type $\tilde E_6$: $\<\underline{\mu}, 2 \rho\> \ge \<\o^\vee_{1, j}, 2 \rho\>=16>2 \times 6$.

Type $\tilde E_7$: $\<\underline{\mu}, 2 \rho\> \ge \<\o^\vee_{7, j}, 2 \rho\>=27>2 \times 7$.

Type $\tilde E_8$: $\<\underline{\mu}, 2 \rho\> \ge \<\o^\vee_{8, j}, 2 \rho\>=58>2 \times 8$.

Type $\tilde F_4$: $\<\underline{\mu}, 2 \rho\> \ge \<\o^\vee_{4, j}, 2 \rho\>=16>2 \times 4$.

Type $\tilde G_2$: $\<\underline{\mu}, 2 \rho\> \ge \<\o^\vee_{2, j}, 2 \rho\>=6>2 \times 2$.

\

Now we come to the classical groups.

\subsubsection{Type $\tilde A_{n-1}$} By applying a suitable automorphism, we may assume that $K_i \subset \tS_i - \{0\}$. By (b) we deduce that (up to isomorphism) one of following cases occurs:

(1) $r = 1$ and $\underline{\mu} \in \{\o_k^\vee;  1 \le k \le n-1\}$;

(2) $r = 1$ and $\underline{\mu} \in \{2\o_1^\vee, \o_1^\vee + \o_{n-1}^\vee, 2\o_{n-1}^\vee\}$;

(3) $r = 2$ and $\underline{\mu} = \o_{1, 1}^\vee + \o_{n-1, 2}^\vee$.

In the last two cases, we have by (b) that $\ssrk(\bG) = \ssrk(\bJ_\t) = n-1$, which means that $\underline{\mu} = \o_1^\vee + \o_{n-1}^\vee$ (which equals $2\o^\vee_1$ if $n=2$) and $\s = \id$ or $\underline{\mu} = \o_{1, 1}^\vee + \o_{n-1, 2}^\vee$ and $\s = {}^1\varsigma_0$.

Now we assume $r = 1$. Suppose $\underline{\mu} = \o_i^\vee$ for some $1 \le i \le n-1$. Notice that $\ssrk(\bJ_1) \le \frac{n}{2} $ (resp. $\ssrk(\bJ_\t) \le \frac{n}{2}$) if $\s \neq \id$ (resp. $\Ad(\t) \circ \s \neq \id$). So either $\underline{\mu} \in \{\o_1^\vee, \o_{n-1}^\vee\}$ or $\underline{\mu}=\o^\vee_2$ with $n=4$.

Suppose $\underline{\mu} = \o_1^\vee$ and $\s=\Ad(\t_k)$ for $0 \le k \neq n-1$. Then by (b) we have $$n-1 = \<\underline{\mu}, 2\rho\> \le \ssrk(\bG)+\ssrk(\bJ_\t) = \gcd(n, k) -1 + \gcd(n, k+1) - 1,$$ which implies $k = n-1$ or $k = 0$. Otherwise, we deduce that (up to isomorphism) $\s = \varsigma_0$.

Suppose $\underline{\mu} = \o^\vee_2$ and $n = 4$. We have (up to isomorphism) $\s = \id$, or $\s = \varsigma_0$, or $\s = \Ad(\t_1)$, or $\s = \Ad(\t_1) \circ \s_0$. The last case does not occur since (b) fails. 

\subsubsection{Type $\tilde B_n$ for $n \ge 3$} Here $\<\o^\vee_{i, j}, 2 \rho\>=i (2n-i)$. Therefore $\<\underline{\mu}, 2 \rho\> \le 2n$ implies that $r = 1$ and $\underline{\mu}=\o^\vee_1$. In this case $\s=\id$ or $\s=\Ad(\t_1)$.

\subsubsection{Type $\tilde C_n$ for $n \ge 2$} Here $$\<\o^\vee_{i, j}, 2 \rho\>=\begin{cases} i (2n-i+1), & \text{ if } i \le n-1; \\ \frac{n(n+1)}{2}, & \text{ if } i=n.\end{cases}$$ Therefore $\<\underline{\mu}, 2 \rho\> \le 2n$ implies that $r = 1$ and either $\underline{\mu}=\o^\vee_1$ or $\underline{\mu}=\o^\vee_n$ with $n \le 3$.

If $\underline{\mu}=\o^\vee_1$, then $\<\underline{\mu}, 2 \rho\>=2n$ and hence $\ssrk(\bJ_\t)=n$. Therefore $\s=id$.

For $n=2$ and $\underline{\mu}=\o^\vee_2$, we have $\s=\id$ or $\s=\Ad(\t_2)$.

For $n=3$ and $\underline{\mu}=\o^\vee_3$, we have $\<\underline{\mu}, 2 \rho\>=6$ and hence $\ssrk(\bJ_\t)=3$. Therefore, $\s=\Ad(\t_3)$ and $\<\underline{\mu}, 2\rho\> = 6 > 4 = \ssrk(\bG) + \ssrk(\bJ_\t)$, contradicting (b).

\subsubsection{Type $\tilde D_n$ for $n \ge 4$} Here $$\<\o^\vee_{i, j}, 2 \rho\>=\begin{cases} i (2n-i-1), & \text{ if } i \le n-2; \\ \frac{n(n-1)}{2}, & \text{ if } i=n-1 \text{ or } n.\end{cases}$$ Therefore $\<\underline{\mu}, 2 \rho\> \le 2n$ implies that $r = 1$ and (up to isomorphism) either $\underline{\mu}=\o^\vee_1$ or $\underline{\mu}=\o^\vee_n$ with $n=5$.

If $n=5$ and $\underline{\mu}=\o^\vee_5$, we have $\<\underline{\mu}, 2 \rho\>=10$ and hence $\ssrk(\bJ_\t)=5$. Therefore $\s=\Ad(\t_4)$. But $\<\underline{\mu}, 2 \rho\> = 10 > 6 = \ssrk(\bG) + \ssrk(\bJ_\t)$, contradicting (b).

If $\underline{\mu}=\o^\vee_1$, then $\<\underline{\mu}, 2 \rho\>=2(n-1)$ and $\ssrk(\bJ_\t) \ge n-2$. Thus we have $\s=\id$, $\s=\s_0$ or $\s=\Ad(\t_1)$. Suppose $\s=\Ad(\t_1)$ and $i \notin K$ for some $1 \le i \le n-1$. Let $J = \tS - \{i, \s(i)\}$. Let $\xi = \underline{\mu}$ if $i = 1$ and $\xi = s_i \cdots s_2 s_1(\underline{\mu})$ if $2 \le i \le n-1$. Then $J_\xi = \emptyset$ if $i \in \{1, n-1\}$ and $J_\xi = \{i+1, \dots, n-1, n\}$ otherwise. Hence (a) fails for the admissible triple $(\xi, J, J)$.

\smallskip

In Table~\ref{tab:1} we list the remaining cases together with the $\s$-orbits and the $\Ad(\t) \circ \s$-orbits on $\tS$. This information will be used in the case-by-case analysis in the remainder of this section.

\bigskip
\begin{table}[h!]
\renewcommand{\arraystretch}{1.7}
\begin{tabular}{@{}l@{\hskip1cm}l@{\hskip1cm}l@{}}
\textbf{Types} & \textbf{$\s$-orbits} & \textbf{$\Ad(\t) \circ \s$-orbits}
\\
\hline
$(\tilde A_{n-1}, \id, \o^\vee_1)$ & $\{i\}, \quad 0 \le i \le n-1$ & $\tS$
\\
\hline
$(\tilde A_{n-1}, \varrho_{n-1}, \o^\vee_1)$ & $\tS$ & $\{i\}, \quad 0 \le i \le n-1$
\\
\hline
$(\tilde A_{2 m}, \varsigma_0, \o^\vee_1)$ & \makecell{$\{0\}, \{i, 2m+1-i\}$,\\ $1 \le i \le m$} & \makecell{$\{m+1\}, \{i, 2m+2-i\}$\\ $1 \le i \le m$\\
{\footnotesize (We set $s_{2m+1}=s_0$.)}}
\\
\hline
$(\tilde A_{2 m+1}, \varsigma_0, \o^\vee_1)$ &
\makecell{$\{0\}$, $\{m+1\}$,\\ $\{i, 2m+2-i\}$\\ $1 \le i \le m$} &
\makecell{$\{i, 2m+3-i\}$,\\ $1 \le i \le m+1$\\
{\footnotesize (We set $s_{2m+2}=s_0$.)}}
\\
\hline
$(\tilde A_{2m+1}, \varrho_{n-1} \circ \varsigma_0, \o_1^\vee)$ &
\makecell{$\{i, 2m+3-i\}$,\\ $1 \le i \le m$\\
{\footnotesize (We set $s_{2m+2}=s_0$.)}}&
\makecell{$\{0\}$, $\{m+1\}$,\\ $\{i, 2m+2-i\}$\\ $1 \le i \le m$}
\\
\hline
\makecell{$(\tilde A_{n-1}, \id, \o^\vee_1+\o^\vee_{n-1})$\\for $n \ge 3$} & $\{i\}, \quad 0 \le i \le n-1$ & $\{i\}, \quad 0 \le i \le n-1$
\\
\hline
$(\tilde A_{n-1} \times \tilde A_{n-1}, {}^1\varsigma_0, (\o^\vee_1, \o^\vee_{n-1}))$ & \makecell{ $\{i\} \sqcup \{i\} \subseteq \tS_1 \sqcup \tS_2$\\for $0 \le i \le n-1$} & \makecell{ $\{i\} \sqcup \{i-1\} \subseteq \tS_1 \sqcup \tS_2$ \\ for $0 \le i \le n-1$}
\\
\hline
$(\tilde A_1, \id, 2 \o^\vee_1)$ & $\{0\}, \{1\}$ & $\{0\}, \{1\}$
\\
\hline
$(\tilde A_3, \id, \o^\vee_2)$ & $\{i\}, \quad 0 \le i \le 3$ & $\{0, 2\}$, $\{1, 3\}$
\\
\hline
$(\tilde A_3, \varsigma_0, \o^\vee_2)$ & $\{0\}$, $\{2\}$, $\{1, 3\}$ & $\{0, 2\}$, $\{1\}$, $\{3\}$
\\
\hline
$(\tilde A_3, \varrho_1, \o_2^\vee)$ & $\tS$ & $\tS$
\\
\hline
$(\tilde B_n, \id, \o^\vee_1)$ & $\{i\}, \quad 0 \le i \le n$ & $\{0, 1\}, \{i\}, \quad 2 \le i \le n$
\\
\hline
$(\tilde B_n, \Ad(\t_1), \o^\vee_1)$ & \makecell{$\{0, 1\}, \{i\}$\\ $2 \le i \le n$} & $\{i\}, \quad 0 \le i \le n$
\\
\hline
$(\tilde C_n, \id, \o^\vee_1)$ & $\{i\}, \quad 0 \le i \le n$ &  $\{i\}, \quad 0 \le i \le n$
\\
\hline
$(\tilde C_2, \id, \o^\vee_2)$ & $\{0\}, \{1\}, \{2\}$ & $\{0, 2\}, \{1\}$
\\
\hline
$(\tilde C_2, \Ad(\t_2), \o^\vee_2)$ & $\{0, 2\}, \{1\}$ & $\{0\}, \{1\}, \{2\}$
\\
\hline
$(\tilde D_n, \id, \o^\vee_1)$ & $\{i\}, \quad 0 \le i \le n$ & \makecell{$\{0, 1\}$, $\{n-1, n\}$,\\  $\{i\}$, $2 \le i \le n-2$}
\\
\hline
$(\tilde D_n, \varsigma_0, \o^\vee_1)$ & \makecell{$\{n-1, n\}$, $\{i\}$\\ $0 \le i \le n-2$} & $\{0, 1\}, \{i\}, \quad 2 \le i \le n$
\\
\hline
\end{tabular}

\bigskip
\caption{}
\label{tab:1}
\end{table}

\subsection{Step (II): Exclude certain $K$}

\subsubsection{$(W_a, \s, \underline{\mu}) = (\tilde A_{2m}, \varsigma_0, \o_1^\vee)$} If $K \neq \tS-\{0\}$, then $K \subseteq J = \tS - \{i, 2m+1-i\}$ for some $1 \le i \le m$. Let $\xi=s_i s_{i-1} \cdots s_1(\underline{\mu})=(0, \cdots, 0, 1, 0, \cdots, 0)$ with the $(i+1)$-th entry equal to $1$. Then $J_\xi=\{i+1, i+2, \ldots, 2m-i\}$. Hence the inequality (a) fails for the admissible triple $(\xi, J, J)$.

\subsubsection{$(W_a, \s, \underline{\mu}) = (\tilde A_{2m+1}, \varsigma_0, \o_1^\vee)$} If $\tS-\{0, m+1\} \nsubseteq K$, then $K \subset J = \tS - \{i, 2m+2-i\}$ for some $1 \le i \le m$. Let $\xi=s_i s_{i-1} \cdots s_1(\underline{\mu})=(0, \cdots, 0, 1, 0, \cdots, 0)$ with the $(i+1)$-th entry equal to $1$. Then $J_\xi=\{i+1, i+2, \ldots, 2m+1-i\}$. Hence the inequality (a) fails for the admissible triple $(\xi, J, J)$.

\subsubsection{$(W_a, \s, \underline{\mu}) = (\tilde A_{2m+1}, \varrho_{n-1} \circ \varsigma_0, \o_1^\vee)$} After applying a suitable inner diagram automorphism, we may assume that $K \subseteq J = \tilde \BS - \{i, 2m+1-i\}$ for some $1 \le i \le m$. Let $\xi = s_i s_{i-1} \cdots s_1(\underline\mu) = (0, \dots, 0, 1, 0 \dots, 0)$ otherwise, where the $(i+1)$-th entry equals $1$. Then $J_\xi =\{i+1, i+2, \cdots, 2m-i\}$ otherwise. Hence the inequality (a) fails for the admissible triple $(\xi, J, J)$.

\subsubsection{$(W_a, \s, \underline{\mu}) = (\tilde A_{n-1}, \id, \o_1^\vee + \o_{n-1}^\vee)$ for $n \ge 3$}

If $|K|<n-1$, then after applying the diagram automorphism $\varsigma_0$, we may assume that $K \subseteq K':= \tS - \{0, i\}$ for some $1 \le i \le n-2$. Let $J=\tS-\{0\}$ and $\xi=s_i s_{i-1} \cdots s_1(\underline{\mu})=(0, \cdots, 0, 1, 0, \cdots, 0, -1)$ with the $(i+1)$-th entry equal to $1$ and the $n$-th entry equal to $-1$. Then $K'_\xi=\{i+1, i+2, \ldots, n-1\}$. Hence the inequality (a) fails for the admissible triple $(\xi, J, K')$.

\subsubsection{$(W_a, \s, \underline{\mu}) = (\tilde A_{n-1} \times \tilde A_{n-1}, {}^1\varsigma_0, \o_{1, 1}^\vee + \o_{n-1, 2}^\vee)$}

If $|K_1| < n-1$, then after applying a suitable inner diagram automorphism, we may assume that $K_1 \subseteq \tS_1 - \{0, i\}$ for some $1 \le i \le n-1$. Let $J = \s(J)$ such that $J_1 = \tS_1 -\{0\}$ and $\xi=(\xi_1, \xi_2)$, where $\xi_2 = (0, \dots, 0, -1)$ and $\xi_2 = (0, \cdots, 0, 1, 0, \cdots, 0)$ with the $(i+1)$-th entry equal to $1$. Then $J_\xi = \{i+1, i+2, \ldots, n-1\}$. Hence the inequality (a) fails for the admissible triple $(\xi, J, J)$.

\subsubsection{$(W_a, \s, \underline{\mu}) = (\tilde A_3, \id, \o_2^\vee)$} Suppose $K$ does not contain two consecutive vertices in the affine Dynkin diagram. Then up to a suitable inner diagram automorphism, we may assume that $K \subset K' := \{1, 3\}$. Let $J = \tS - \{0\}$ and $\xi = s_2(\underline{\mu}) = (1, 0, 1, 0)$. Then the inequality (a) fails for the admissible triple $(\xi, J, K')$.

\subsubsection{$(W_a, \s, \underline{\mu}) = (\tilde A_3, \varsigma_0, \o_2^\vee)$}

Suppose $K \subseteq K' := \{1, 3\}$. Let $J = \tS - \{0\}$ and $\xi = s_2(\underline{\mu}) = (1, 0, 1, 0)$. Then the inequality (a) fails for the admissible triple  $(\xi, J, K')$.

Suppose $K \subseteq J := \{0, 2\}$. Let $\xi = s_1 s_2(\underline{\mu}) = (0, 1, 1, 0)$. Then $J_\xi = \emptyset$ and hence the inequality (a) fails for the admissible triple  $(\xi, J, J)$.

\subsubsection{$(W_a, \s, \underline{\mu}) = (\tilde A_3, \Ad(\t_1), \o_2^\vee)$}

The semisimple rank of $\bJ_\t$ is zero and the inequality (b) fails.

\subsubsection{$(W_a, \s, \underline{\mu}) = (\tilde B_n, \id, \o_1^\vee)$}

Suppose $K \subseteq J := \tS - \{i\}$ with $2 \le i \le n-1$. Let $\xi = s_i s_{i-1} \dots s_1(\underline{\mu}) = (0, \dots, 0, 1, 0, \dots, 0)$ with the $(i+1)$-th entry being $1$. Then $J_\xi = \{i+1, i+2, \dots, n\}$ and hence the inequality (a) fails for the admissible triple $(\xi, J, J)$.

\subsubsection{$(W_a, \s, \underline{\mu}) = (\tilde B_n, \Ad(\t_1), \o_1^\vee)$}

Suppose $K \subseteq J := \tS - \{i\}$ with $2 \le i \le n-1$. Let $\xi = s_i s_{i-1} \dots s_1(\underline{\mu}) = (0, \dots, 0, 1, 0, \dots, 0)$ with the $(i+1)$-th entry being $1$. Then $J_\xi = \{i+1, i+2, \dots, n\}$ and hence the inequality (a) fails for the admissible triple $(\xi, J, J)$.

Suppose $K \subseteq J := \tS - \{0, 1\}$. Then $J_{\underline{\mu}} = \emptyset$ and hence the inequality (a) fails for the admissible triple $(\underline{\mu}, J, J)$.

\subsubsection{$(W_a, \s, \underline{\mu}) = (\tilde C_n, \id, \o_1^\vee)$}

If $K \subseteq J := \tS - \{i\}$ for $1 \le i \le n-1$. Let $\xi = s_i s_{i-1} \dots s_1(\underline{\mu}) = (0, \dots, 0, 1, 0, \dots, 0)$ with the $(i+1)$-th entry being $1$. Then $J_\xi = \{i+1, i+2, \dots, n\}$ and hence the inequality (a) fails for the admissible triple $(\xi, J, J)$.

\subsubsection{$(W_a, \s, \underline{\mu}) = (\tilde C_2, \id, \o_2^\vee)$}

Suppose $K \subseteq K' := \{1\}$. Let $J = \{1, 2\}$ and $\xi = s_2(\underline{\mu})$. Then the inequality (a) fails for the admissible triple $(\xi, J, K')$.

\subsubsection{$(W_a, \s, \underline{\mu}) = (\tilde C_2, \Ad(\t_2), \o_2^\vee)$}
Suppose $K \subseteq J := \{1\}$. Then $J_{\underline{\mu}} = \emptyset$ and hence the inequality (a) fails for the admissible triple $(\underline{\mu}, J, J)$.

\subsubsection{$(W_a, \s, \underline{\mu}) = (\tilde D_n, \id, \o_1^\vee)$}

Suppose $K \subseteq J := \tS - \{i\}$ for $2 \le i \le n-1$. Let $\xi = s_i s_{i-1} \dots s_1(\underline{\mu}) = (0, \dots, 0, 1, 0, \dots, 0)$ with the $(i+1)$-th entry being $1$. Then $J_\xi = \{i+1, i+2, \dots, n\}$ and hence the inequality (a) fails for the admissible triple $(\underline{\mu}, J, J)$.

Suppose $K \subseteq K' := \tS - \{0, 1\}$. Let $J = \tS - \{0\}$ and $\xi = s_1(\underline{\mu})$. Then the inequality (a) fails for the admissible triple $(\underline{\mu}, J, K')$.

\subsubsection{$(W_a, \s, \underline{\mu}) = (\tilde D_n, \varsigma_0, \o_1^\vee)$}

Suppose $K \subseteq J := \tS - \{i\}$ for $2 \le i \le n-2$. Let $\xi = s_i s_{i-1} \dots s_1(\underline{\mu}) = (0, \dots, 0, 1, 0, \dots, 0)$ with the $(i+1)$-th entry being $1$. Then $J_\xi = \{i+1, i+2, \dots, n\}$ and hence the inequality (a) fails for the admissible triple $(\underline{\mu}, J, J)$.

Suppose $K \subseteq K' := \tS - \{0, 1\}$. Let $J = \tS - \{0\}$ and $\xi = s_1(\underline{\mu})$. Then the inequality (a) fails for the admissible triple $(\underline{\mu}, J, K')$.

Suppose $K \subseteq J := \tS - \{n-1, n\}$. Let $\xi = (s_{n-1} \cdots s_1(\underline{\mu}) = (0, \dots, 0, 1)$. Then $J_\xi = \emptyset$ and hence the inequality (a) fails for the admissible triple $(\xi, J, J)$.

\subsection{Step (III): Final verification}

To finish the proof of Theorem~\ref{thm:characterization-coxeter-type}, it only remains to check in each case that the set $\KAdm(\mu)_0$ is the set given in Table~\ref{table:enhanced-coxeter-data}. From the explicit description, we see that these cases are of Coxeter type. We omit the explicit computations.

\section{Consequences for RZ spaces}\label{sec:rz-spaces}


In this section, we explain consequences of our results for Rapoport--Zink
spaces. Also compare the paper~\cite{Wang4} by H.~Wang for applications to
Shimura varieties.

\subsection{Definitions}

We consider the situation where $F = \mathbb Q_p$,
and where the pair $(\bG, \mu)$ corresponds to a Rapoport--Zink space. As before, we consider the basic case, i.e., we denote by $b$ the basic $\s$-conjugacy class in $B(\bG, \mu)$.

Since there are different constructions of RZ spaces in the PEL case and the more general case of Hodge type, we will axiomatize the properties that we require, rather than fixing one of the constructions.

In all cases, the RZ space $\mathcal M(\bG, \mu, b)_K$ is a formal scheme over $\breve O$, the ring of integers of $\breve F$. We denote by $\kk$ the residue class field of $\breve F$, and by a subscript $-_{\kk}$ indicate the base change to $\kk$.

Note that the results in the previous sections are group-theoretic in nature and hence concern parahoric level structures, but the known constructions of RZ spaces work for stabilizers of facets in the Bruhat--Tits building. See the discussion in Section~\ref{parahoric-vs-stabilizer}. To take this possible difference into account right from the beginning, we change notation as follows: We denote by $\Pstab$ the stabilizers of facets of the base alcove, and by $\Pcirc$ the corresponding parahoric.

The setup of the theory entails that $\Pstab$ and $\Pcirc$ are always defined over $F$, i.e., fixed by $\s$.

We denote by $\Gr_{\Pcirc}$ the partial affine flag variety for $\Pcirc$ (i.e., with $\kk$-valued points $\breve{G}/\Pcirc$), and similarly by $\Gr_{\Pstab}$ the ``partial affine flag variety'' for $\Pstab$ with $\kk$-valued points $\breve{G}/\Pstab$. Let $\pi\colon \Gr_{\Pcirc}\to \Gr_{\Pstab}$ denote the projection.

Since $\Pcirc \subseteq \Pstab$ is a normal subgroup, the (finite) quotient group $\Pstab/\Pcirc$ acts on $\Gr_{\Pcirc}$ on the right by $g\Pcirc \cdot p = gp\Pcirc$, $p\in \Pstab/\Pcirc$, and $\Gr_P$ is the quotient by this action. Let $\breve{\kappa}\colon \brG\to \pi_0(\Gr_{\Pcirc}) = \pi_1(\bG)_{\Gamma_0}$ be the Kottwitz homomorphism. Since $\Pcirc$ is the kernel of $\breve{\kappa}_{|\Pstab}$, and correspondingly, the parahoric group scheme corresponding to $\Pcirc$ is the connected component of the ``stabilizer group scheme'' corresponding to $\Pstab$, the restriction of $\pi$ to any connected component of $\Gr_{\Pcirc}$ is an isomorphism onto a connected component of $\Gr_{\Pstab}$; cf.~\cite[Thm.~1.4, App.~Prop.~3]{Pappas-Rapoport:Twisted}. In other words, $\pi$ identifies those connected components which are mapped to each other by $P$.

In this way, we obtain a perfect scheme $\Gr_{\Pstab}$ with $\kk$-valued points $\brG/\Pstab$ and such that the projection $\pi\colon \Gr_{\Pcirc}\to \Gr_{\Pstab}$ is an isomorphism when restricted to a connected component of $\Gr_{\Pcirc}$.

We write
\[
    X_{\Pstab} := X(\mu, b)_{\Pstab} = \{ g\in \Gr_{\Pstab};\ g^{-1}b\sigma(g) \in P \Adm(\mu) P\},
\]
which also inherits the structure of a perfect scheme.

Similarly as in~\cite{GHN}, we consider the following condition:

\bigskip
($\diamondsuit$)
For facet stabilizers $\Pstab\subset \Pstab'$, we have a projection $\mathcal M(\bG, \mu,
b)_\Pstab\to \mathcal M(\bG, \mu, b)_{\Pstab'}$, and there are isomorphisms
\[
    \mathcal M(\bG, \mu, b)_{\Pstab, \kk}^{p^{-\infty}} \cong X(\mu, b)_\Pstab
\]
of perfect schemes, compatible with the projections for inclusions $\Pstab\subset \Pstab'$.
\bigskip

The second condition that we need to impose is the following compatibility
between the RZ spaces for levels $\Pstab_i$ and the RZ space attached to
the intersection $\Pstab:=\bigcap_i \Pstab_i$.
It follows from property~($\diamondsuit$) that the morphism
$\mathcal M(\bG, \mu, b)_{K, \kk}^{\rm red}\to\prod_i \mathcal M(\bG, \mu, b)_{K_i, \kk}^{\rm red}$ is a homeomorphism onto a closed subscheme of its target. We will impose the following stronger statement as our second axiom:

\bigskip
($\clubsuit$) The natural morphism $\mathcal M(\bG, \mu, b)_{\Pstab, \kk}^{\rm red}\to \prod_i \mathcal M(\bG, \mu, b)_{\Pstab_i, \kk}^{\rm red}$ is a closed immersion of $\kk$-schemes.
\bigskip

\subsection{The PEL case}
For most RZ spaces of PEL type, it is known that the assumptions ($\diamondsuit$) and ($\clubsuit$) are satisfied.

Consider an RZ space attached to a PEL datum as in~\cite[Ch.~3]{Rapoport-Zink},
see also~\cite{Hartwig} and~\cite{Rapoport-Viehmann} for summaries and further
discussions. Let $\bG$ and $\mu$ be the group and cocharacter attached to it. In
this context, the level structure is given by a polarized chain of lattices
(i.e., we use a lattice model for the relevant Bruhat--Tits building). Denote by
$\Pstab \subset \breve G$ the stabilizer of this fixed standard chain.

Denote by $\mathcal M^{\rm naive}_\Pstab$ the corresponding RZ space
(\cite[Def.~3.21]{Rapoport-Zink}) over $O_{\breve E}$, the ring of integers of
the completion of the maximal unramified extension of the local reflex field
$E$. It depends on the choice of a ``framing object'' $\mathbf X$ (a $p$-divisible group with additional structure corresponding to the group $\mathbf G$) and parameterizes pairs $((X_\Lambda), (\rho_\Lambda))$, where $(X_\Lambda)$ is a chain of isogenies of $p$-divisible groups (over a scheme $S$ on which $p$ is locally nilpotent) indexed by the fixed lattice chain, and $\rho_\Lambda$ is a quasi-isogeny between $X_\Lambda$ and $\mathbf X$ over the closed subscheme $V(p)\subseteq S$. These data are required to satisfy certain compatibilities, see~\cite[Def.~3.12]{Rapoport-Zink}.

It is clear that property ($\diamondsuit$) cannot in general be expected to
hold for the ``naive'' RZ spaces, because their special fiber in general
comprises strata not reflected in the set $\Adm(\mu)$.  Therefore we pass to the
corresponding ``flat RZ space''.

As shown in~\cite{Rapoport-Zink}, the space $\mathcal M^{\rm naive}_\Pstab$
admits a \emph{local model diagram}
\[
    \mathcal M^{\rm naive}_\Pstab \longleftarrow \widetilde{\mathcal M^{\rm naive}_\Pstab} \longrightarrow M^{{\rm naive}, \wedge}_\Pstab,
\]
where $M^{\rm naive}_\Pstab$ denotes the local model
of~\cite[Def.~3.27]{Rapoport-Zink} and $-^\wedge$ denotes the $p$-adic
completion. See~\cite[Ch.~3]{Rapoport-Zink}, in particular Sections~3.26--3.35.
Denote by $M^{\rm flat}_\Pstab$ the flat closure inside $M^{\rm naive}_\Pstab$ of its generic fiber.

Then $\mathcal M_\Pstab$ is defined by pulling back and pushing forward the inclusion $M^{\rm flat}_\Pstab \subseteq M^{\rm naive}_\Pstab$ along the local model diagram (after passing to the completion along the special fiber), i.e., we have a diagram
\[
    \xymatrix{
        \mathcal M_\Pstab \ar[d]& \ar[l]\widetilde{\mathcal M}_\Pstab \ar[r]\ar[d] & M^{{\rm loc}, \wedge}_\Pstab\ar[d]\\
        \mathcal M^{\rm naive}_\Pstab & \ar[l]\widetilde{\mathcal M^{\rm naive}_\Pstab} \ar[r] & M^{{\rm naive}, \wedge}_\Pstab
}
\]
where the vertical arrows are closed immersions and both squares are cartesian. Then $\mathcal M_\Pstab$ is flat over $O_{\breve E}$.

Forgetting the endomorphism and polarization structure, we obtain a closed embedding into the corresponding RZ space for the general linear group and the same lattice chain, now considered without additional structure. We will denote this space by $\mathcal M'_{\Pstab'}$ (imitating the notation of~\cite{Hamacher-Kim}), in particular $P'$ is the stabilizer of our lattice chain inside the general linear group $G' = GL(\mathbf N)(\brQp)$ of automorphisms of the rational Dieudonn\'e module $\mathbf N$ of $\mathbf X$.

By Dieudonn\'e theory, we obtain a commutative diagram of inclusions
\begin{equation}\label{eqn:embedding-GL}
    \xymatrix{
        \mathcal M_\Pstab(\kk) \ar[r] & \mathcal M^{\rm naive}_\Pstab(\kk) \ar[r]\ar[d] & \mathcal M'_{\Pstab'}(\kk) \ar[d]\\
                                      & \brG/\Pstab \ar[r] & G'/\Pstab'
    }
\end{equation}
Here the right vertical map maps a point $(X_\bullet, \rho)\in \mathcal M'_{\Pstab'}$, where $X_\bullet$ is a chain of isogenies of $p$-divisible groups indexed by the fixed periodic lattice chain, and $\rho$ is a quasi-isogeny with the framing object $\mathbb X$, to $g\Pstab'$ where $g$ maps the fixed (partial) standard lattice chain to the chain of Dieudonn\'e modules of $X_\bullet$ (considered inside the rational Dieudonn\'e module of $\mathbb X$ via $\rho$). Since in the $GL_n$ case the vertical map is induced by an isomorphism of perfect schemes onto its image (cf.~\cite[Prop.~3.11]{Zhu} which easily generalizes to general parahoric level structure in that case), the same is true for this map.

The map $\mathcal M^{\rm naive}_\Pstab(\kk)\to \mathcal M'_{\Pstab'}(\kk)$ which forgets the additional structure is an inclusion, because the additional structure is uniquely determined (by that structure on $\mathbf X$ and the quasi-isogenies $\rho_\Lambda$), if it exists. Cf.~\cite[Proof of Thm.~3.25]{Rapoport-Zink}.

The Frobenius morphism on $\mathbf N$ is given by $b\s$ for some $b\in \brG$.

Consider a point $(\mathcal F_\Lambda)_\Lambda\in M^{\rm naive}(\kk)$. By definition, each $\mathcal F_\Lambda$ is a subspace of $\Lambda\otimes_{\mathbb Z_p}\kk$ (with further properties which we do not state here; in comparison to~\cite[Def.~3.27]{Rapoport-Zink} we switch from quotients to subspaces). Equivalently, we can record this data as a lattice $\tilde{\mathcal F}$ lying between $\Lambda$ and $p\Lambda$. This defines an inclusion $M^{\rm naive}(\kk) \to \brG/P$, and the action of $\Pstab$ on $\brG/\Pstab$ on the left preserves the subset $M^{\rm naive}(\kk)$.

By the definition of the local model diagram we obtain a commutative diagram
\[
    \xymatrix{
        \mathcal M_P^{\rm naive}(\kk) \ar[r]\ar[d] & \Pstab\backslash M^{\rm naive}(\kk)\ar[d]\\
        \brG/P \ar[r] & \Pstab \backslash \brG / \Pstab
    }
\]
where the lower horizontal map maps $g\Pstab$ to the double coset of $g^{-1}b\s(g)$. Likewise, $\Pstab$ acts on $M_\Pstab^{\rm loc}(\kk)\subseteq M^{\rm naive}(\kk)$ since this action comes from an action of the (smooth) stabilizer group scheme associated with $P$ on the $O_{\breve E}$-scheme $M^{\rm naive}$. We write $\mathcal A(\mu)_\Pstab = \Pstab\backslash M^{\rm loc}(\kk) \subset \Pstab\backslash \brG / \Pstab$.
Similarly, we write $\Adm(\mu)_\Pstab = \Pstab\backslash \Pstab\Adm(\mu)\Pstab/\Pstab$.


\begin{proposition}\label{prop:diamond-for-PEL}
    If $\mathcal A(\mu)_\Pstab = \Adm(\mu)_\Pstab$, then assumption ($\diamondsuit$) is satisfied. More precisely, the inclusion $\mathcal M_P(\kk) \subset \breve G/\Pstab$ is induced by an isomorphism
    \[
        \mathcal M_{\Pstab, \kk}^{p^{-\infty}} \cong X(\mu, b)_\Pstab
    \]
    of perfect schemes, and these isomorphisms are compatible with the projections for passing to sub-lattice chains.
\end{proposition}

\begin{proof}
    Since we know (from the embedding into a $GL$ situation) that the map is
    induced from a morphism of perfect schemes, it is enough to check the claim
    on $\kk$-valued points.

    The above discussion shows, together with our assumption $\mathcal A(\mu)_\Pstab = \Adm(\mu)_\Pstab$ and the definition of $M^{\rm loc}_\Pstab$, that we have a diagram
    \[
        \xymatrix{
            \mathcal M_\Pstab(\kk) \ar[r]\ar[d] & \Pstab\backslash M^{\rm loc}(\kk) \ar[r]^{=}\ar[d] & \Adm(\mu)_\Pstab\ar[d] \\
    \mathcal M^{\rm naive}_\Pstab(\kk) \ar[r] & \Pstab\backslash M^{\rm naive}(\kk) \ar[r] & \Pstab\backslash\brG/\Pstab
}
    \]
    in which both squares are cartesian. Viewing $\mathcal M^{\rm naive}_\Pstab(\kk)\subset \brG/\Pstab$ as before, the lower horizontal map is given by $g\Pstab\mapsto \Pstab g^{-1}b\s(g)\Pstab$. So we see that $\mathcal M_\Pstab(\kk) \subseteq X(\mu, b)_\Pstab$, and that it is enough to show that $X(\mu, b)_\Pstab\subseteq \mathcal M^{\rm naive}_\Pstab(\kk)$ in order to complete the proof.

    For this inclusion, note that the square in the diagram~\eqref{eqn:embedding-GL}, while not cartesian in general, is close to being cartesian. More precisely, $\mathcal M^{\rm naive}_\Pstab(\kk)$ is defined inside the intersection $\mathcal M'_{\Pstab'}(\kk)\cap \brG/\Pstab$ by imposing the Kottwitz determinant condition. Since this condition can be checked on the local model, and since it is satisfied by definition on $M^{\rm naive}$, and a fortiori on $M^{\rm loc}$, it is enough to show that $X(\mu, b)_\Pstab\subseteq \mathcal M'_{\Pstab'}$.

Denote by $\mu'$ the composition of $\mu$ with the inclusion $\mathbf G\to GL(\mathbf N)$. Since the identification of $\mathcal M'_{\Pstab'}(\kk)$ with the corresponding generalized affine Deligne-Lusztig variety $X^{GL}(\mu', b)_{\Pstab'}$ for the general linear group is easy to check, we see that is is enough to show that $X_P$ embeds into $X^{GL}(\mu', b)_{\Pstab'}$ under the embedding $\breve G/\Pstab\subseteq G'/\Pstab'$. This follows once we can prove that $\Adm(\mu)_\Pstab$ embeds into the admissible set $\Adm(\mu')_{\Pstab'}\subset \Pstab'\backslash G'/\Pstab'$.

    This compatibility of admissible sets follows from the inclusions $M^{\rm loc}_\Pstab(\kk) \subseteq M^{\rm naive}_\Pstab(\kk) \subseteq M^{GL}_{\Pstab'}(\kk)$, where $M^{GL}_{\Pstab'}$ denotes the local model for the general linear group, using once more the assumption $\mathcal A(\mu)_\Pstab = \Adm(\mu)_\Pstab$ and the fact that for the general linear group the corresponding equality is known, as well.

    It is clear that everything above is compatible with the projections arising
    from forgetting some of the lattices in our chain.
\end{proof}



\begin{remark}\label{rmk:axioms-for-PEL}
    The condition $\mathcal A(\mu)_\Pstab = \Adm(\mu)_\Pstab$ is known to hold in many cases. Note that almost the same condition is posed as Axiom~3.2 in~\cite{HR:axioms}; the only difference is that in our setting we can (and need to) be a little bit more precise as to how this identification arises.
    \begin{enumerate}
        \item
            Assume that $p$ is odd, $\bG/\mathbb Q_p$ is connected and splits over
            a tamely ramified extension and that the stabilizer $\Pstab$ of our
            lattice chain is a parahoric subgroup. Then by the work of Pappas
            and Zhu~\cite[Thm.~1.1, Thm.~1.2]{Pappas-Zhu}, in many individual
            cases, its predecessors), the special fiber of $M^{\rm flat}_\Pstab$
            is the union of Schubert varieties (in an equal characteristic
            affine Grassmannian) indexed by the admissible set
            $\Adm(\mu)_\Pstab$. See also loc.~cit., Section~8.2.
        \item
            The condition has been checked in many individual cases, including
            cases where the stabilizer $\Pstab$ is not parahoric. Specifically,
            see Smithling's papers~\cite{Smithling:ram-unit-odd},
            \cite{Smithling:ram-unit-even} for ramified unitary groups, and
            ~\cite{Smithling:orth} for split even orthogonal groups.
    \end{enumerate}
\end{remark}

\begin{proposition}\label{prop:club-for-PEL}
    For RZ spaces  $\mathcal M_\Pstab$ of PEL type the projections in condition {\rm ($\diamondsuit$)} exist and condition ($\clubsuit$) is satisfied.
\end{proposition}

\begin{proof}
    It is clear that for $P\subseteq P'$ there is a projection morphism between the corresponding RZ spaces.
Now, for the ``naive'' versions the definition in terms of chains of $p$-divisible groups shows immediately that the morphism in question is a monomorphism. Since the irreducible components of the source are proper, cf.~\cite[Prop.~2.32]{Rapoport-Zink}, and the source is locally of finite type over $\kk$ (\cite[Thm.~3.25]{Rapoport-Zink}), it follows that the morphism is a closed immersion. The ``flat'' RZ spaces are closed formal subschemes of the naive RZ spaces, so the above property continues to hold.
\end{proof}

\subsection{RZ spaces of Hodge type}

In the work of Kim~\cite{Kim} and Hamacher and Kim~\cite{Hamacher-Kim} where RZ
spaces for data of Hodge type are constructed, the bijection $\mathcal M(\bG, \mu,
b)_{\Pstab}(\kk) \cong X(\mu, b)_\Pstab(\kk)$ on $\kk$-valued points is an
essential feature of the construction, see~\cite[Prop.~4.3.5]{Hamacher-Kim}.
Zhu (\cite[Prop.~3.11]{Zhu}) proved that this set-theoretical equality implies
the above isomorphism of perfect schemes, using results of Gabber and Lau on
Dieudonn\'e theory over perfect rings to handle the case $\bG=GL_n$, and then
embedding the general situation into a suitable $GL_n$-situation. While
in~\cite{Zhu} it was assumed that $K$ is hyperspecial, the only reason for this
assumption is that at the time of writing RZ spaces of Hodge type had been
constructed only in this special situation; the paper~\cite{Hamacher-Kim}
appeared only later.

Note however that the previous paragraph concerns only the situation for a fixed
level $\Pstab$. Because of the way RZ spaces are defined in the Hodge type
situation, it is not clear that these bijections are compatible
with the projection maps attached to a pair $\Pstab\subset \Pstab'$ of parahoric
subgroups. In fact, the definition relies on embedding the situation into an RZ
space of Siegel type (i.e., associated with a group of symplectic similitudes
and hyperspecial level structure). However, it is not evident whether the result
is independent of the choice of embedding, and it does not seem clear whether
such embeddings can be chosen in a compatible way given $\Pstab\subset \Pstab'$.

\subsection{The weak Bruhat-Tits stratification}
\label{subsec:weak-bt-rz}

Next we discuss the question of defining a (weak) Bruhat--Tits stratification on Rapoport--Zink spaces. Hence, we now restrict to the fully Hodge--Newton decomposable case. Then we have the weak Bruhat--Tits stratification on $X(\mu, b)_{\Pcirc}$ (Section~\ref{subsec:bt-stratification}).

Since property ($\diamondsuit$) involves $\Pstab$ instead of $\Pcirc$, we first need to discuss the question of defining a weak BT stratification on $X(\mu, b)_\Pstab$. To this end, we impose the following additional assumption:

\bigskip
($\heartsuit$)\qquad The projection $X_{\Pcirc} \to X_\Pstab$ is surjective.
\bigskip

See Section~\ref{parahoric-vs-stabilizer} for two sufficient criteria for this
assumption. Of course, it is trivially satisfied if $\Pstab=\Pcirc$. It is also
satisfied in those cases that have been studied in detail on the Shimura variety
side (attached to ramified unitary groups). We do not know of an example where
this property fails.

Let us denote by $X_{\Pcirc, c}$ the ``component'' indexed by
$c\in\pi_0(\Gr_{\Pcirc})$, i.e., $X_{\Pcirc, c} = X_{\Pcirc}\cap {\breve
\kappa}^{-1}(c)$. As explained in Section~\ref{parahoric-vs-stabilizer},
$X_{\Pcirc, c} = \emptyset$ unless $c\in \pi_0(\Gr_{\Pcirc})^\sigma$.

For each $c\in \pi_0(\Gr_{\Pcirc})^\sigma$, the projection $X_{\Pcirc}\to
X_{\Pstab}$ restricts to an isomorphism
\[
    \gamma_c\colon X_{\Pcirc, c}  \isoarrow X_{\Pstab, \pi(c)}.
\]

Here we denote by $\pi(c)$ the connected component of $\Gr_{\Pstab}$ which is
the image of $c$, and by $X_{\Pstab, \pi(c)}$ its intersection with
$X_{\Pstab}$. Then $\pi^{-1}(X_{\Pstab, \pi(c)})$ is the union of the
$X_{\Pcirc, c'}$ where $c'$ ranges over the $\Pstab/\Pcirc$-orbit of $c$ in
$\pi_0(\Gr_{\Pcirc})$.

To establish that $\gamma_c$ above is an isomorphism, recall that each connected
component of $\Gr_{\Pcirc}$ maps isomorphically onto a connected component of
$\Gr_{\Pstab}$. Therefore the only question is the surjectivity, which follows
from ($\heartsuit$).

Since all BT strata in $X_{\Pcirc}$ are connected, each stratum lies in one of
the components $X_{\Pcirc, c}$. Using the isomorphisms $\gamma_c$, we obtain
a Bruhat--Tits stratification on $X_{\Pstab}$ which is independent of the choice
of $c$, as the following lemma shows:

\begin{lemma}
    Let $c, c'\in \pi_0(\Gr_{\Pcirc})^\sigma$ such that $c'c^{-1}\in \Pstab/\Pcirc$. Let $j Y(w) \subseteq X_{\Pcirc, c'}$ be a BT stratum. Then
    \[
        \gamma_{c}^{-1}(\gamma_{c'}(jY(w))) \subseteq X_{\Pcirc, c}
    \]
    is a BT stratum.
\end{lemma}

\begin{proof}
    The map $\gamma_{c}^{-1}\circ \gamma_ {c'}$ is given by $g\Pcirc\mapsto g(c')^{-1}c\Pcirc$.
    We may represent $(c')^{-1}c$ by (a representative of) a length $0$ element in $\tW$ which comes from $\Pstab$. To simplify the notation, we change notation and denote this element by $c$. Then we have $\s(c) = c$ inside $\tW$ (passing to a different representative, if necessary we could achieve the same property on the level of elements of $\breve G$).

    If $g\in jY(w)$, then $g^{-1}\t\s(g)\in \Pcirc \cdot_{\sigma} (\brI w\t \brI)$. But then
    \[
        c^{-1} g^{-1} \t\s(g)\s(c) \in c^{-1}\Pcirc \cdot_\s(\brI w\t \brI) \s(c)
        = \Pcirc \cdot_\s (\brI c^{-1}w\t \s(c) \brI).
    \]
    Now $c^{-1}w\t \s(c) = c^{-1} w\t c$ is again an EKOR element for $\Pcirc$ (i.e., it lies in ${}^K\tW\cap \Adm(\mu)$, where $K\subset\tS$ denotes the set of simple affine reflections generating $\Pcirc$).

    Since the element $c^{-1} w\t c$ is independent of $g$, the lemma follows.
\end{proof}

We hence get a well-defined ``Bruhat--Tits stratification'' on $X_{\Pstab}$ and the projection $X_{\Pcirc}\to X_{\Pstab}$ is compatible with the stratifications:

\begin{corollary}
    In the fully Hodge--Newton decomposable case, the isomorphisms $X_c \to \pi(X_c)$, for a connected component $c$ of $\Gr_{\Pcirc}$, define a weak Bruhat--Tits stratification on $X_P$ (i.e., the stratification is independent of the choices of $c$).
\end{corollary}

Consider an inclusion $\Pstab\subset \Pprimestab$ of facet stabilizers and the corresponding inclusion $\Pcirc \subset \Pprimecirc$. We obtain a commutative diagram
\[
    \xymatrix{
        X_{\Pcirc} \ar[r]\ar[d] & X_{\Pprimecirc}\ar[d]\\
        X_{\Pstab} \ar[r] & X_{\Pprimestab}.
    }
\]
The sets of connected components of $\Gr_{\Pcirc}$ and $\Gr_{\Pprimecirc}$ coincide. Fixing a connected component $c$ of $\Gr_{\Pcirc}$ and denoting by $c'$ its image in $\Gr_{\Pprimecirc}$, we can restrict the above diagram to $c$ and $c'$, and obtain a diagram
\[
    \xymatrix{
        X_{\Pcirc, c} \ar[r]\ar[d] & X_{\Pprimecirc, c'}\ar[d]\\
        X_{\Pstab, \pi(c)} \ar[r] & X_{\Pprimestab, \pi'(c')}
    }
\]
(with $\pi'\colon \Gr_{\Pprimecirc}\to \Gr_{\Pprimestab}$ the projection) where the vertical morphisms are isomorphisms. Since the upper horizontal map maps each BT stratum in $X_c$ isomorphically onto a BT stratum in $X_{c'}$, the same holds true for the lower horizontal map.

We also record the following easy lemma:

\begin{lemma}
    Let $\mathbf f$, $\mathbf f_i$, $i\in I$, be facets of the base alcove (viewed as simplices, i.e., as subsets of the set of vertices of the base alcove). Denote by $\Pstab$, $\Pstab_i$ the facet stabilizers, and by $\Pcirc$, $\Pcirc_i$ the corresponding parahoric subgroups. The following are equivalent:
    \begin{enumerate}
        \item $\mathbf f = \bigcup_i \mathbf f_i$,
        \item $\Pstab = \bigcap_i \Pstab_i$,
        \item $\Pcirc = \bigcap_i \Pcirc_i$.
    \end{enumerate}
\end{lemma}

\begin{proof}
    Clearly, (1) and (2) are equivalent. Furthermore, (3) and (1) are equivalent since a parahoric subgroup is determined by the set of affine simple reflections which it contains.
\end{proof}

Now we can invoke property ($\diamondsuit$) and obtain that in the fully Hodge--Newton decomposable case and under the above assumptions, the spaces $\mathcal M(\bG, \mu, b)_{\Pstab, \kk}^{p^{-\infty}}$ carry a weak BT stratification which is compatible with the weak BT stratification on $X_P$ via ($\diamondsuit$). For an inclusion $\Pstab\subset \Pprimestab$, the corresponding map of RZ spaces maps each BT stratum in the source isomorphically (in the sense of perfect schemes) onto a BT stratum in the target.

Since the morphism $\mathcal M(\bG, \mu, b)_{K, \kk}^{p^{-\infty}} \to \mathcal
M(\bG, \mu, b)_{K, \kk}^{\rm red}$ is a homeomorphism, we likewise get

\begin{corollary}
    In the fully Hodge--Newton decomposable case and under the above assumptions, the spaces $\mathcal M(\bG, \mu, b)_{\Pstab, \kk}$ carry a \emph{weak Bruhat--Tits stratification} into locally closed reduced subschemes which is compatible with the weak BT stratification on $X_P$ via passing to perfections and the isomorphism {\rm ($\diamondsuit$)}. For an inclusion $\Pstab\subset \Pprimestab$, the corresponding map of RZ spaces maps each BT stratum in the source homeomorphically onto a BT stratum in the target.

The perfection of each stratum is isomorphic to the perfection of a classical Deligne--Lusztig variety.
\end{corollary}

The goal of this section is to investigate the following property (under the assumption of full Hodge--Newton decomposability):

\medskip
\textbf{Property} $BT_\Pstab$: In the (weak) Bruhat--Tits stratification of
$\mathcal M(\bG, \mu, b)_{\Pstab, \kk}^{\rm red}$, each stratum is isomorphic to
the corresponding classical Deligne--Lusztig variety, without passing to the perfection. The closure of each stratum is isomorphic to the closure of this classical Deligne--Lusztig variety in the corresponding finite-dimensional partial flag variety.

\medskip\noindent
Our result will be that if ($BT_{\Pstab_i}$) holds for all $i$, then ($BT_\Pstab$) also holds for $\Pstab=\cap_i \Pstab_i$. Note that we need to use in the proof that we already know the result after perfection for level $\Pstab$.

For most maximal parahoric level structures $\Pstab$, the property ($BT_P$) has been established by now, and this result will allow us to deduce the same property for many non-maximal level structures, and in particular for most level structures of Coxeter type. See Section~\ref{sec:individual-cases} for a discussion of the individual cases.

\begin{proposition}\label{prop:bt-for-intersection}
    Let $(\bG, \mu)$ be fully Hodge--Newton decomposable. Assume that properties
    {\rm ($\diamondsuit$)}, {\rm ($\clubsuit$)} and {\rm ($\heartsuit$)} hold.
    As above, let $\Pstab_i$ be facet stabilizers (of facets of the base alcove which are fixed by $\s$) such that
    $BT_{\Pstab_i}$ holds for all $i$. Then $BT_\Pstab$ holds for $\Pstab := \bigcap_i \Pstab_i$.
\end{proposition}

\begin{proof}
    To shorten the notation, we write $\mathcal M_\Pstab$ for $\mathcal M(\bG, \mu,
    b)_{\Pstab, \kk}^{\rm red}$.
    The assumption implies that $\mathcal M_\Pstab$ has a Bruhat--Tits stratification as explained above.
    Consider a BT stratum $S\subset \mathcal M_\Pstab^{p^{-\infty}}$ (i.e., $S$ is
    isomorphic to the perfection of a classical Deligne--Lusztig variety $X$). We denote by $Y$ the reduced locally closed subscheme inside $\mathcal M_\Pstab$ corresponding to $S$.

    For each $i$, $S$ maps isomorphically onto a BT stratum $S_i \subset \mathcal M_{\Pstab}^{p^{-\infty}}$. So $S_i$ is the perfection of the same DL variety $X$. By the assumption $BT_{\Pstab_i}$, the stratum $Y_i \subset \mathcal M_{\Pstab_i}$ corresponding to $S_i$ is isomorphic to $X$.

Now consider the closed embedding $\mathcal M_\Pstab \to \prod_i \mathcal
M_{\Pstab_i}$ given by Property~($\clubsuit$). Restricting to $Y$, we obtain a closed embedding $Y\to \prod_i Y_i$. Passing to perfections we obtain an embedding $S\to \prod_i S_i$.

    Then the reduced closed subscheme of $\prod_i Y_i$ with underlying space $S$
    is isomorphic to the Deligne--Lusztig variety $X$: Both are reduced locally
    closed subschemes, so it is enough to check that they have the same
    underlying topological space. But this can be checked on the perfections,
    where we know the result from the group-theoretic situation.

    This implies that the reduced subscheme $Y$ of $\mathcal M_{\Pstab}$ with same
    topological space as $S$ is isomorphic to $X$, as desired. Exactly the same reasoning works when we replace each of $S$, $Y$, $S_i$, $Y_i$ with its closure.
\end{proof}

\subsection{Stabilizers versus parahoric subgroups}\label{parahoric-vs-stabilizer}

It remains to discuss in more detail in which cases the assumption ($\heartsuit$) is satisfied, i.e., when the map $X_{\Pcirc}\to X_{\Pstab}$ is surjective. Let us briefly recall the setup.

Let $\Pstab\subset \breve G$ be the stabilizer of a (poly-)simplex in the Bruhat--Tits building. The intersection $\Pcirc$ of $\Pstab$ with the kernel of the Kottwitz homomorphism $\breve{\kappa}\colon \breve G\to \pi_0(L\breve{G}) = \pi_1(\bG)_{\Gamma_0}$ is a parahoric subgroup (and all parahoric subgroups arise in this way), but in general, $\Pcirc \subset \Pstab$ is a proper normal subgroup.

\begin{remark}
This phenomenon occurs in several cases:
\begin{enumerate}
    \item[(i)]
        The case of ramified unitary groups has been discussed in detail by Pappas and Rapoport in~\cite{Pappas-Rapoport:Twisted} Section~4 and \cite[\S\S1.2, 1.3]{Pappas-Rapoport:LM3}.
    \item[(ii)]
        For even special orthogonal groups (with hyperspecial level structure), see Yu's notes~\cite[Example~2.2.3.1]{Yu:Lectures-buildings}.
    \item[(iii)]
        Applying restriction of scalars along an unramified extension to the above examples, one can construct further examples, and in particular obtains examples where the action of $\s$ on $\Pstab/\Pcirc$ is not trivial.
\end{enumerate}
If $\bG$ is semi-simple and simply connected, or more generally if $\pi_0(L\breve{G})$ is torsion-free, then the finite subgroup $\Pstab/\Pcirc \subseteq \pi_0(L\breve{G})$ must be trivial, so that we have $\Pcirc = \Pstab$.
\end{remark}

In the sequel, we assume that $\Pcirc$ is a standard parahoric subgroup, or in other words, that $\Pstab$ is the stabilizer of a face of the base alcove.
As mentioned above, the group-theoretic results obtained in this paper concern a priori the ``parahoric setting'', but RZ spaces have been defined, so far, in the ``stabilizer setting''. See also Remark~5.4 in the paper~\cite{Hamacher-Kim} by Hamacher and Kim.

Throughout the following discussion, we fix $\mu$ and $b\in B(\bG, \mu)$, and we write $X_{\Pcirc} = X(\mu, b)_{\Pcirc}$, and similarly write
\[
    X_{\Pstab} := X(\mu, b)_{\Pstab} = \{ g\in \Gr_{\Pstab};\ g^{-1}b\sigma(g) \in P \Adm(\mu) P\}.
\]
We have
\[
    X_{\Pcirc}\subseteq \pi^{-1}(X_{\Pstab}).
\]

\begin{lemma}[{cf.~\cite[Rmk.~5.4]{Hamacher-Kim}}]
    \begin{enumerate}
        \item[(i)]
            If $v \in W_a$, $\omega \in \Omega$ with $v\omega \in P$, then $\omega\in P$. (Here we do not distinguish between elements of $\tW$ and representatives of these elements in $\breve G$.)
        \item[(ii)]
            For $\omega \in \Omega$, we have $\omega \Adm(\mu) \omega^{-1} = \Adm(\mu)$.
        \item[(iii)]
            We have $\Pstab \Adm(\mu) \Pstab = \Pcirc \Adm(\mu) \Pstab$.
    \end{enumerate}
\end{lemma}

\begin{proof}
    To prove part (1), say that $\Pstab$ is the stabilizer of the face $\mathbf f$ of the base alcove. Under the assumption that $v\omega$ stabilizes $\mathbf f$, we need to show that the same is true for $\omega$. Assume that $\omega\mathbf f \ne \mathbf f$. Since $\omega$ preserves the base alcove, this means that $\omega\mathbf f$ is a face of a type different from the type of $\mathbf f$. Because the action of $W_a$ preserves the type of faces, it is then impossible that $v\omega\mathbf f = \mathbf f$.

    For (2), note that conjugation preserves the orbit $W_0(\mu)$ and conjugation by length $0$ elements preserves the Bruhat order.

    Now we prove part (3). By part (1), we can write $P = \bigcup_p \Pcirc p$, where $p$ runs through a system of representatives of $\Pstab/\Pcirc$ given by (representatives of) length $0$ elements in $\tW$. Since $\Pcirc\subset \Pstab$ is normal, the claimed statement follows from (2).
\end{proof}

Now for $g\Pcirc\in\Gr_{\Pcirc}$, we have
\[
    g\Pcirc \in X_{\Pcirc} \quad \Longleftrightarrow \quad g^{-1}b\sigma(g) \in \Pcirc \Adm(\mu) \Pcirc,
\]
and
\[
    g\Pcirc \in \pi^{-1}(X_{\Pstab}) \quad \Longleftrightarrow \quad g^{-1}b\sigma(g) \in \Pstab \Adm(\mu) \Pstab.
\]
The lemma shows that $\Pstab \Adm(\mu) \Pstab = \Pcirc \Adm(\mu) \Pstab$, and denoting by $\breve{G}_1$ the kernel of the Kottwitz homomorphism $\breve{\kappa}$, we obtain that
\[
    \Pcirc \Adm(\mu) \Pcirc = \Pstab \Adm(\mu) \Pstab \cap b\breve{G}_1.
\]
Hence an element $g\Pcirc \in \pi^{-1}(X_{\Pstab})$ lies in $X_{\Pcirc}$ if and only if $\breve{\kappa}(g^{-1}b\sigma(g)) = \breve{\kappa}(b)$, or equivalently $\breve{\kappa}(g) = \sigma(\breve{\kappa}(g))$ (since $\pi_0(\Gr_{\Pcirc})$ is abelian). In other words,
\[
    X_{\Pcirc} = \pi^{-1}(X_{\Pstab}) \cap \breve{\kappa}^{-1}(\pi_0(\Gr_{\Pcirc})^\sigma).
\]

\medskip
Choosing a (set-theoretic) section of the projection $\pi_0(\Gr_{\Pcirc}) \to \pi_0(\Gr_{\Pstab}) = \pi_0(\Gr_{\Pcirc})/(\Pstab/\Pcirc)$, we obtain a section $\iota\colon \Gr_{\Pstab} \to \Gr_{\Pcirc}$ of $\pi$, and then can identify
\[
    \Gr_{\Pcirc} = \bigsqcup_{p\in\Pstab/\Pcirc} \iota(\Gr_{\Pstab})p,
\]
i.e., $\Gr_{\Pcirc}$ is isomorphic to a disjoint union of copies of $\Gr_{\Pstab}$.

\begin{proposition}\label{prop:product-decomposition}
    Assume that we can decompose $\pi_0(\Gr_{\Pcirc})$ as a product $\Pstab/\Pcirc \times C$ for some subgroup $C$, such that $\s$ preserves this product decomposition, and fix such an identification $\pi_0(\Gr_{\Pcirc}) = \Pstab/\Pcirc \times C$.

    This choice defines a section $\iota\colon \Gr_{\Pstab}\to \Gr_{\Pcirc}$ of $\pi$, and an identification
    \[
        \Gr_{\Pcirc} = \bigsqcup_{p\in \Pstab/\Pcirc} \iota(\Gr_{\Pstab})p.
    \]

    With respect to this decomposition, the space $X_{\Pcirc}$ is a disjoint union of copies of $X_{\Pstab}$. More precisely,
    \[
        X_{\Pcirc} = \bigsqcup_{p\in (\Pstab/\Pcirc)^\sigma} \iota(X_{\Pstab})p.
    \]
    In particular, in this situation assumption {\rm ($\heartsuit$)} is satisfied.
\end{proposition}

\begin{proof}
    Fix a decomposition $\pi_0(\Gr_{\Pcirc})=\Pstab/\Pcirc \times C$. This gives us a section $\pi_0(\Gr_{\Pstab}) \cong C \to \pi_0(\Gr_{\Pcirc})$ to which we apply the above discussion. We then have an isomorphism $\bigsqcup_{c\in C} (\Gr_{\Pcirc})_c \to \Gr_{\Pstab}$. Here $(\Gr_{\Pcirc})_c = \breve{\kappa}^{-1}(c)$ denotes the connected component corresponding to $c$. We obtain a section $\iota\colon \Gr_{\Pstab}\to \Gr_{\Pcirc}$ of $\pi$, which in turn gives us the identification
    \[
        \Gr_{\Pcirc} = \bigsqcup_{p\in \Pstab/\Pcirc} \iota(\Gr_{\Pstab})p.
    \]

    Now fix $p\in \Pstab/\Pcirc$ and let $g\Pcirc\in \pi^{-1}(X_P) \cap \iota(\Gr_{\Pstab})p$. We have $\breve{\kappa}(g) = (p, c) \in \Pstab/\Pcirc \times C$ for some $c\in C$. Then $g\Pcirc\in \pi^{-1}(X_P)$ gives us $\breve{\kappa}(g^{-1}\sigma(g))\in \Pstab/\Pcirc$, i.e., $c = \sigma(c)$. Thus the condition $\breve{\kappa}(g) = \sigma(\breve{\kappa}(g))$ is equivalent to $p=\sigma(p)$.
\end{proof}

Let us also investigate when we have equality $X_{\Pcirc} = \pi^{-1}(X_{\Pstab})$.

\begin{lemma}
    The following conditions are equivalent:
    \begin{enumerate}
        \item
            For all $x\in \pi_0(\Gr_{\Pcirc})$ with $x - \sigma(x) \in \Pstab/\Pcirc$, we have $x = \sigma(x)$.
        \item
            We have that $\s$ fixes all elements of $\Pstab/\Pcirc$, and that the sequence
            \[
                0 \to \Pstab/\Pcirc \to (\pi_0(\Gr_{\Pcirc}))^\s \to C^\s \to 0
            \]
            is exact, where $C$ denotes the quotient of $\pi_0(\Gr_{\Pcirc})$ by $\Pstab/\Pcirc$.
    \end{enumerate}
\end{lemma}

\begin{proof}
    This is a purely group-theoretic reformulation which uses only that $\Pstab/\Pcirc$ is a $\s$-invariant subgroup of the abelian group $\pi_0(\Gr_{\Pcirc})$.
\end{proof}

Note that the conditions in the lemma are satisfied for example in the following cases:
\begin{enumerate}
    \item[(a)]
        The group $\pi_0(\Gr_{\Pcirc})$ is a direct product $\pi_0(\Gr_{\Pcirc}) = C\times \Pstab/\Pcirc$ for some subgroup $C$, the operation of $\sigma$ preserves this product decomposition, and all elements of $\Pstab/\Pcirc$ are fixed by $\sigma$.
    \item[(b)]
        The action of $\sigma$ on $\pi_0(\Gr_{\Pcirc})$ is trivial. (For instance, this holds if $\bG$ splits over a totally ramified extension of $F$.)
\end{enumerate}

\begin{proposition}\label{prop:cartesian}
    If the equivalent conditions of the previous lemma are satisfied, then the diagram
    \[
        \xymatrix{ X_{\Pcirc} \ar[r]\ar[d] & \Gr_{\Pcirc} \ar[d]^{\pi}\\
            X_{\Pstab}\ar[r] & \Gr_{\Pstab}
        }
    \]
    is cartesian, i.e., $X_{\Pcirc} = \pi^{-1}(X_{\Pstab})$.
    In particular, in this situation assumption {\rm ($\heartsuit$)} is satisfied.
\end{proposition}

\begin{proof}
    As the above discussion shows, we need to show that $\breve{\kappa}(g) = \sigma(\breve{\kappa}(g))$ for all $g\Pcirc\in \pi^{-1}(X_{\Pstab})$. But those $g$ satisfy $\breve{\kappa}(g) - \sigma(\breve{\kappa}(g))\in \Pstab/\Pcirc$, so this claim follows immediately from (1) in the above lemma.
\end{proof}

\section{Known results, new results and open cases}\label{sec:individual-cases}

As in the previous section, we assume that $(\bG, \mu)$ comes from a Rapoport--Zink space.

\subsection{Discussion of individual cases}\label{subsec:known-new-open}~\\

Table~\ref{table:known-new-open} lists the cases that come from Shimura
varieties and where $\mu$ is non-central in each $\overline{F}$-factor of $\bG$.
As we stated before, it has been checked in many cases that the strata of the
Bruhat--Tits stratification are classical Deligne--Lusztig varieties (also
before passing to perfections), and that the closures of strata are isomorphic to the closures of these classical Deligne--Lusztig varieties in the corresponding finite-dimensional partial flag varieties. The strategy in the papers cited below consists
of the following steps: Define a set-theoretic bijection in terms of Dieudonn\'e
theory, extend it to a morphism of schemes using Zink's display theory, and
check that one obtains an identification of schemes using Zariski's main theorem
and the normality of Deligne--Lusztig varieties. In all cases, the
stratification is given by viewing the Bruhat--Tits building of the group $\bJ$
in terms of lattices (``vertex lattices''), and defining the strata by
considering relative positions of Dieudonn\'e modules and such vertex lattices.
This means that Dieudonn\'e theory gives a suitable way to set up the
identifications ($\diamondsuit$), and that then the BT stratification as defined
in this paper coincides with the stratifications defined in the papers mentioned
below.

\begin{table}
	\renewcommand{\arraystretch}{1.7}
    \begin{tabularx}{\textwidth}{llc}
        $(\bG, \mu)$ & maximal case(s)\qquad\qquad\qquad\qquad\phantom{.} & minimal case \\
		\hline
        \makecell{\textbf{Linear, unitary groups} (families)}\\\hline
    $(A_n, \o^\vee_1, \emptyset)$ &  \multicolumn{2}{c}{\knownresult all cases: $\dim X(\mu, \t)_K = 0$} \\
    $({}^d A_{d-1}, \o^\vee_1, \emptyset)$ & \multicolumn{2}{c}{\knownresult all cases: \cite{Drinfeld}, \cite[Thm.~3.72]{Rapoport-Zink}} \\
    $({}^2 A'_{2m}, \o^\vee_1, \BS)$ & \knownresult \cite{Vollaard-Wedhorn} & ($=$ maximal) \\
    $({}^2 A'_{2m+1}, \o^\vee_1, \tilde \BS-\{s_0, s_{m+1}\})$ & \knownresult \cite{Vollaard-Wedhorn} & \newresult, see~\ref{even-unram-unitary}  \\
    \makecell{$(A_{n-1} \times A_{n-1}, (\o^\vee_1, \o^\vee_{n-1}), {}^1 \varsigma_0, \BS \sqcup \BS)$\\for $n\ge 2$} & \knownresult \cite{Stamm}, \cite{Helm-Tian-Xiao}, \cite{Tian-Xiao}, \cite{Tian-Xiao-2}
                                                                                                                       & ($=$ maximal) \\[-.2cm]
\makecell{$(\Res_{E/F}(A_{n-1}), (\o^\vee_1,  \o^\vee_{n-1}), \BS)$\\for $n \ge 3$, $E/F$ ramified quadratic} & \opencase
    & ($=$ maximal) \\[-.2cm]
    $(B$-$C_n, \o^\vee_1, \tilde \BS-\{s_0, s_n\})$ & \knownresult $\tS-\{ s_n\}$: \cite{Rapoport-Terstiege-Wilson}, $\tS-\{ s_0\}$: \cite{Wu} & \newresult, see~\ref{ramified-unitary-even}\\
    $(C$-$BC_n, \o^\vee_1, \tilde \BS-\{s_0, s_n\})$ & \knownresult $\tS-\{ s_n\}$: \cite{Rapoport-Terstiege-Wilson}, $\tS-\{ s_0\}$: \cite{Wu} & \newresult, see~\ref{ramified-unitary-odd} \\
    $({}^2 B$-$C_n, \o^\vee_1, \tilde \BS-\{s_n\})$ & \knownresult \cite{Rapoport-Terstiege-Wilson} & ($=$ maximal)\\
\hline
          \makecell{\textbf{Orthogonal groups} (families)}\\\hline
    $(B_n, \o^\vee_1, \tilde \BS-\{s_0, s_n\})$ & \knownresult $\tS-\{s_0\}$: \cite{howard-pappas2}, \opencase $\tS-\{s_n\}$ & \opencase\\
    $({}^2 B_n, \o^\vee_1, \tilde \BS-\{s_n\})$ & \opencase & \opencase \\
    $(C$-$B_n, \o^\vee_1, \tilde \BS-\{s_0, s_n\})$ & \opencase & \opencase \\
        $(D_n, \o^\vee_1, \tilde \BS-\{s_0, s_n\})$  & \knownresult \cite{howard-pappas2} & \newresultA, see \ref{split-orthogonal} \\
        $({}^2 D_n, \o^\vee_1, \tilde \BS-\{s_n\})$ & \opencase & \opencase \\\hline
        \textbf{Exceptional cases}\\\hline
    \makecell{$(\Res_{E/F}(A_{1}), (\o^\vee_1,  \o^\vee_{1}), \emptyset)$\\$E/F$ ramified quadratic} & \knownresult \cite{Bachmat-Goren} & \opencase \\[-.2cm]
        $(A_3, \o^\vee_2, \{s_1, s_2\})$ & \knownresult \cite{Fox} & \newresult, see~\ref{split-U22} \\
    $({}^2 A_3', \o^\vee_2, \BS)$ & \knownresult \cite{howard-pappas} & ($=$ maximal) \\
    $(C_2, \o^\vee_2, \{s_0\})$ & \knownresult \cite{Kaiser} & \knownresult \cite{goertz-yu-1} \\
    $(C$-$B_2, \o^\vee_2, \{s_0\})$ & \opencase & \opencase \\
    $(C$-$BC_2, \o^\vee_2, \{s_0\})$ & \opencase & \opencase \\
    $(C$-$BC_2, \o^\vee_2, \{s_2\})$ & \opencase & \opencase \\
    $({}^2 C_2, \o^\vee_2, \{s_0, s_2\})$ & \knownresult \cite{Oki}, \cite{Wang1} & ($=$ maximal)\\
        $({}^2 C$-$B_2, \o^\vee_2, \{s_0, s_2\})$ & \opencase & \opencase \\
        \makecell{$(\Res_{E/F}(C$-$BC_1), (\o^\vee_1, \o^\vee_1), \emptyset)$\\$E/F$ ramified quadratic} & \opencase & \opencase \\\hline
        \multicolumn{2}{l}{\textbf{Cases not arising from a Shimura variety}}\\\hline
		$(C_n, \o^\vee_1, \tilde \BS-\{s_0, s_n\})$ \\
	\end{tabularx}
    \caption{}\label{table:known-new-open}
\end{table}

The meaning of the symbols used in the table is as follows:

\vspace{.5cm}
\begin{tabular}{ll}
    \knownresult & a known case,\\[.3cm]
    \opencase &  \makecell{an open case,}\\[.3cm]
    \newresult & \makecell{a case newly settled by this paper (Prop.~\ref{prop:bt-for-intersection}),}\\[.3cm]
    \newresultA & \makecell{a case newly settled by this paper (Prop.~\ref{prop:bt-for-intersection}), up to the verification of the\\assumptions ($\diamondsuit$), ($\clubsuit$), ($\heartsuit$) for the involved RZ spaces.}
\end{tabular}

Below we give further remarks on some of the cases.

\subsubsection{$(\tilde{A}_{n-1}, \id, \omega_1^\vee, \emptyset)$ --- Harris--Taylor type}\label{HT-type}

The automorphism $\tau$ acts by rotation $$0\mapsto 1\mapsto 2\mapsto \dots \mapsto n-1\mapsto 0$$ on the affine Dynkin diagram. We have $\dim X(\mu, \tau)_K= 0$ for all $K\subsetneq \tS$, so the set-theoretic bijection defining the Bruhat--Tits stratification is automatically an isomorphism before passing to the perfection. This case arises from unitary Shimura varieties attached to groups that split over $\mathbb Q_p$. Cf.~the book~\cite{HT} by Harris and Taylor.

\subsubsection{$(\tilde{A}_{n-1}, \varrho_{n-1}, \omega_1^\vee, \emptyset)$ --- Drinfeld case}\label{example:drinfeld}

The automorphism $\tau$ is the same as in Section~\ref{HT-type}, so the composition $\Ad(\tau)\circ\sigma$ is the identity. The only rational level structure is $K=\emptyset$. We have $\dim X(\mu, \tau)= n-1$, and in this case, the set $B(\bG, \mu)$ has only one element: The basic locus equals the whole Shimura variety. The description given by Drinfeld (\cite{Drinfeld}, see also~\cite[Thm.~3.72]{Rapoport-Zink} and the subsequent discussion) of the RZ space as a formal scheme (that can be constructed by gluing pieces indexed by simplices in the Bruhat--Tits building of the corresponding group $\bJ$ (split of type $A_{n-1}$)) shows in particular that property $BT_\emptyset$ holds.

The maximal elements in $\Adm(\mu)$ are
\[
    \tau s_{n-1} s_{n-2} \cdots s_1,\quad \tau s_{n-2} s_{n-3} \cdots s_0,\quad \ldots,\quad \tau s_0 s_{-1} \cdots s_{-(n-2)}.
\]
    The automorphism $\Ad(\t)\circ \s$ is the identity map on $\tS$, and in particular the $\s$-support of an element is simply the support of $w\t\i$. From this description one sees that each element $w\in \Adm(\mu)$ is determined by its $\s$-support $\supp_\s(w)$, and that for all $K$ the order $\le_{K,\s}$ (see Section~\ref{subsec:bt-stratification}) coincides with the Bruhat order on ${}^K\tW$.

The individual strata of the BT stratification are isomorphic to classical Deligne--Lusztig varieties in products of general linear groups. In fact, for each $w$ the ambient group is the reductive group with Dynkin diagram given by $\supp_\s(w)$ (i.e., all other vertices and all edges that involve one of those are discarded) with Frobenius action given by $\Ad(\t)\circ \s = \id$.

For each $w$, the index set for the strata of type $w$ is given by $\bJ(F)/(\bJ(F)\cap \mathcal P_w)$, where $\mathcal P_w$ is the standard parahoric subgroup generated by $\supp_\s(w) \sqcup I(K, w, \s)$, and where $I(K, w, \s)$ is the subset of $K$ comprising all those vertices which are not connected to $\supp_\s(w)$, cf.~Lemma~\ref{descr-IKws}.

\subsubsection{$(\tilde{A}_{2m}, \varsigma_0, \omega_1^\vee, \tS - \{ s_0\})$ --- odd unramified unitary group case}

The automorphism $\tau$ is the same as in Section~\ref{HT-type}, so the composition $\Ad(\tau)\circ\sigma$ acts as the reflection
\[
    1 \leftrightarrow 0,\quad 2 \leftrightarrow 2m,\quad \ldots
\]
The only level structure of Coxeter type, $K=\tS - \{ s_0\}$, is hyperspecial, and it was shown in~\cite{Vollaard-Wedhorn} that property $BT_K$ holds.

\subsubsection{$(\tilde{A}_{2m+1}, \varsigma_0, \omega_1^\vee, \tS - \{ s_0, s_m\})$ --- even unramified unitary group case}\label{even-unram-unitary}

The automorphism $\tau$ is the same as in Section~\ref{HT-type}, so the composition $\Ad(\tau)\circ\sigma$ acts as the reflection
\[
    1 \leftrightarrow 0,\quad 2 \leftrightarrow 2m+1,\quad \ldots
\]
In this case, $\sigma$ fixes two vertices in the affine Dynkin diagram, namely $0$ and $m$. Both level structures $\tS - \{ s_0\}$ and  $\tS - \{ s_m\}$ are hyperspecial, and we can again apply the results of~\cite{Vollaard-Wedhorn} to see that property $BT_K$ holds for $K$ hyperspecial. It then follows from Proposition~\ref{prop:bt-for-intersection} that the same is true for $K=\tS - \{ s_0, s_m\}$. Note that Propositions~\ref{prop:diamond-for-PEL} and~\ref{prop:club-for-PEL} ensure that ($\diamondsuit$) and ($\clubsuit$) hold in this case. The sets $\mathcal A(\mu)_P$ and $\Adm(\mu)_P$ consist only of the double coset of $t^\mu$. Since $\Pstab$ is a parahoric subgroup, ($\heartsuit$) is satisfied trivially.

Let us make the case of level structure $\tS - \{ s_0, s_m\})$ more explicit: As listed in Table~\ref{table:enhanced-coxeter-data}, we have
\[
    \KAdm_0 = \{ s_{[2m+2, 2m+2-i]}s_{[m+1, m+1-j]}\t;\ i,j\ge -1, i+j\le m-2 \},
\]
so this index set depends on two parameters $i, j$ and in particular an element of $\KAdm_0$ is not determined by its length in general. (As before we set $s_{2m+2} := s_0$.)

For $w = s_{[2m+2, 2m+2-i]}s_{[m+1, m+1-j]}\t$, we have
\[
    \supp_\s(w) = \{ 0, 1, 2, \dots, i+1,  2m+2-i, \dots, 2m+2-1 \} \sqcup \{m+1-j, m+2-j, \dots, m + j + 1 \}.
\]
Note that the condition $i+j \le m-2$ ensures that this is a proper subset of $\tS$. We see that each element is determined by its $\s$-support. We write the $\s$-support as a disjoint union of two intervals (possibly empty, depending on the choice of $i$ and $j$) which are disconnected in the Dynkin diagram. The individual strata of the BT stratification are classical Deligne--Lusztig varieties in the group (over the finite residue class field of $F$) specified by the Dynkin diagram $\supp_\s(w)$, i.e., a product of two unitary groups (or just one if one of the intervals is empty).

For each $w$, the index set for the strata of type $w$ is given by $\bJ(F)/(\bJ(F)\cap \mathcal P_w)$, where $\mathcal P_w$ is the standard parahoric subgroup generated by $\supp_\s(w) \sqcup I(K, w, \s)$, and where $I(K, w, \s)$ can be described as in Lemma~\ref{descr-IKws}.

\begin{remark}
    Cho~\cite{Cho} studied the unramified unitary group case for non-hyperspecial maximal parahoric level structure. This case is fully HN decomposable, but not Coxeter type.
\end{remark}



\subsubsection{$(\tilde{A}_{3}, \id, \omega_2^\vee, \tS - \{ s_0, s_3\})$}\label{split-U22}

This case corresponds to a ``split $U(2, 2)$''. Note that for each $i\in\mathbb Z / 4\mathbb Z$, $(\tilde{A}_{3}, \id, \omega_2^\vee, \tS - \{ s_i, s_{i+1}\})$ is isomorphic to the above datum. For every $i$, the level structure $\tS - \{ s_i\}$ is hyperspecial and for those cases it was shown by Fox~\cite{Fox} that the Bruhat--Tits strata are isomorphic to Deligne--Lusztig varieties before perfection.

Note that Propositions~\ref{prop:diamond-for-PEL} and~\ref{prop:club-for-PEL} ensure that ($\diamondsuit$) and ($\clubsuit$) hold in this case. It is well-known that the sets $\mathcal A(\mu)_P$ and $\Adm(\mu)_P$ coincide in this case. Since $\Pstab$ is a parahoric subgroup, ($\heartsuit$) is satisfied trivially.
Hence Proposition~\ref{prop:bt-for-intersection} can be applied.

\subsubsection{$(B$-$C_n, \o^\vee_1, \tilde \BS-\{s_0, s_n\})$  --- even ramified unitary group case}\label{ramified-unitary-even}

To apply Proposition~\ref{prop:bt-for-intersection}, we need to check that the assumptions ($\diamondsuit$), ($\clubsuit$) and ($\heartsuit$) are satisfied. For the first two, we can use Propositions~\ref{prop:diamond-for-PEL} and~\ref{prop:club-for-PEL}. In fact, it follows from Smithling's paper~\cite{Smithling:ram-unit-even} that $\mathcal A(\mu)_\Pstab = \Adm(\mu)_\Pstab$. It follows from Proposition~\ref{prop:product-decomposition} that ($\heartsuit$) is satisfied.

\subsubsection{$(C$-$BC_n, \o^\vee_1, \tilde \BS-\{s_0, s_n\})$ --- odd ramified unitary group case}\label{ramified-unitary-odd}

To apply Proposition~\ref{prop:bt-for-intersection}, we need to check that the assumptions ($\diamondsuit$), ($\clubsuit$) and ($\heartsuit$) are satisfied. For the first two, we can use Propositions~\ref{prop:diamond-for-PEL} and~\ref{prop:club-for-PEL}. In fact, it follows from Smithling's paper~\cite{Smithling:ram-unit-odd} that $\mathcal A(\mu)_\Pstab = \Adm(\mu)_\Pstab$. In this case, $\Pstab$ is a parahoric subgroup, hence ($\heartsuit$) is satisfied.

\subsubsection{$(C$-$BC_2, \o^\vee_2, \{s_0\})$}

In terms of the affine Weyl group, this is the case $(\tilde{C}_2, \id, \o^\vee_2)$.
Let $K = \{ s_0\}$, one of the two minimal level structures of Coxeter type (the other one being $\{ s_2\}$ -- while the two enhanced Coxeter data are isomorphic, this is not true for the enhanced Tits data, because for them we must take the orientation of the local Dynkin diagram into account). We have
\[
    \KAdm(\mu)_0 = \{ \t, s_1\t, s_2\t \}.
\]
The automorphism $\Ad(\t) = \Ad(\t)\circ \s$ is given by interchanging $s_0$ and $s_2$, and fixing $s_1$. The strata are points (for $w=\t$), and classical Deligne--Lusztig varieties in $SL_2$ (for $w=s_1\t$, which has $\supp_\s(w)= \{ s_1\}$) and in the restriction of scalars of $SL_2$ along a quadratic extension (for $w=s_2\t$, which has $\supp_\s(w)=\{s_0, s_2\}$), respectively.

For each $w$, the index set for the strata of type $w$ is given by $\bJ(F)/(\bJ(F)\cap \mathcal P_w)$, where $\mathcal P_w$ is the standard parahoric subgroup generated by $\supp_\s(w)$, since in this case,  $I(K, w, \s) = \emptyset$ for all $w\in \KAdm(\mu)_0$ (cf.~Lemma~\ref{descr-IKws}).

\begin{remark}
    In~\cite{Wang2}, Wang studied the ${}^2C_2$ case with Iwahori level structure.
\end{remark}

\subsubsection{$(D_n, \o^\vee_1, \tilde \BS-\{s_0, s_n\})$}\label{split-orthogonal}

Since this case is not of PEL type, it is less clear that the RZ spaces for the different level structures (as defined by Kim~\cite{Kim}, Howard and Pappas~\cite{howard-pappas2} in the hyperspecial case, and by Hamacher and Kim~\cite{Hamacher-Kim} in general) are compatible in the sense of assumption~($\clubsuit$) above. Once this has been established, the result of Howard and Pappas together with Proposition~\ref{prop:bt-for-intersection} implies the result for level structure $\tilde \BS-\{s_0, s_n\}$.

\section{Smoothness of closures of strata}
\label{sec:smoothness}

In this section, we study the smoothness of the closures of strata.

Studying the smoothness requires us to consider the actual schemes rather than perfect schemes. Therefore the discussion below applies
\begin{itemize}
    \item in the equal characteristic case, and
    \item in cases of RZ spaces where Property $BT_K$ is satisfied (Section~\ref{subsec:weak-bt-rz}).
\end{itemize}
In those cases, the closures of strata are isomorphic to closures of (fine)
Deligne--Lusztig varieties in certain partial flag varieties as recalled in the
next paragraphs, and the whole discussion can take place in this setting.

Recall the Bruhat--Tits stratification as in Section~\ref{subsec:bt-stratification}.
Fix a Coxeter type case $(\bG, \mu, K)$ and an element $w\in \KCox(\mu)$. Write $w':= w\t\i$.

Recall from Section~\ref{subsec:bt-stratification} that we denote by $\mathcal P^\flat_w$ the standard parahoric subgroup generated by $\supp_\s(w)$. The quotient $\mathcal P^\flat_w/\brI$ is naturally identified with the full flag variety $\overline{G}/\overline{B}$ for the maximal reductive quotient $\overline{G}$ of the special fiber of the parahoric group scheme attached to $\mathcal P_w^\flat$ and its standard Borel $\overline{B}$. Let $\overline{W}$ be the Weyl group of $\overline{G}$ (for the maximal torus induced from our choice of maximal torus in $\bG$). The automorphism $\Ad(\t)\circ \s$ induces an automorphism on $\overline{G}$, $\overline{W}$, etc., which we will denote by $\bar \s$. The Dynkin diagram of $\overline{G}$ is $\supp_\s(w)=\supp_{\bar \s}(w')$ (i.e., we keep from the affine Dynkin diagram of $\bG$ the vertices lying in $\supp_\s(w)$ and the edges involving only vertices in this subset). 

The parahoric subgroup generated by $\supp_\s(w)\cap K$ induces a parabolic subgroup $\overline{Q}\subseteq \overline{G}$. By abuse of notation, we also denote by $\overline{Q}$ the corresponding set of simple reflections inside $\overline{W}$ so that we have $\overline{W}_{\overline{Q}}$, the set ${}^{\overline{Q}}\overline{W}$ of minimal length representatives, etc.

The strata of type $w$ are isomorphic to the ``fine'' Deligne--Lusztig variety
\[
    Y(w) \cong \{ g\in \overline{G}/{\overline{Q}};\ g\i \bar \s(g)\in {\overline{Q}}\cdot_{\bar \s} {\overline{B}}w'{\overline{B}} \}.
\]
The isomorphism between the BT stratum and $Y(w)$ extends to an isomorphism between the closure of the stratum and the closure of $Y(w)$ inside $\overline{G}/{\overline{Q}}$.

\subsection{Relating Deligne--Lusztig varieties and Schubert varieties}

We first show that the singularities in the closure of $Y(w)$ in $\overline{G}/{\overline{Q}}$ are smoothly equivalent to singularities in a certain Schubert variety. To do so, we need to show that the closure of $Y(w)$ in $\overline{G}/{\overline{Q}}$ equals the closure of a certain ``coarse'' Deligne--Lusztig variety. It was already explained in~\cite{GH} how to do so, but for the sake of completeness we repeat the short argument here.

Let
\[
    X_{\overline{Q}}(w') = \{ g{\overline{Q}};\ g\i \bar \s(g)\in {\overline{Q}}w' \bar \s({\overline{Q}})\}.
\]
We have $Y(w) \subseteq X_{\overline{Q}}(w')$, and hence we get the same inclusion for the closures. Because $w'$ is an $\bar \s$-Coxeter element in $\overline{W}$, $X_Q(w')$ is irreducible. To show the equality $\overline{Y(w)} = \overline{X_{\overline{Q}}(w')}$ of the two closures, it is therefore enough to show that $\dim Y(w)$ equals $\dim X_{\overline{Q}}(w')$, i.e., that
\[
    \ell(w) = \ell(\max(\overline{W}_{\overline{Q}}w'\overline{W}_{\bar \s({\overline{Q}})})) - \ell(w_{{\overline{Q}}, 0}),
\]
where $w_{0, {\overline{Q}}}$ denotes the longest element of $\overline{W}_{\overline{Q}}$.

Since $\ell(w) = \ell(w_{{\overline{Q}},0}w') - \ell(w_{{\overline{Q}},0})$, this is equivalent to saying that $w_{{\overline{Q}}, 0}w'$ is the longest element of $\overline{W}_Qw'\overline{W}_{\bar \s(Q)}$, or equivalently that $w'$ is the longest element of ${}^{\overline{Q}} \overline{W} \cap \overline{W}_{\overline{Q}}w'\overline{W}_{\bar \s({\overline{Q}})}$.

The truth of this final statement can be established by a case-by-case check for all cases in Table~\ref{table:enhanced-coxeter-data}, which we omit here. 

The closure $\overline{X_{\overline{Q}}(w')}$ is smoothly equivalent to the closure of ${\overline{Q}}w' \bar \s({\overline{Q}}) /\bar \s({\overline{Q}}) \subseteq \overline{G}/\bar \s({\overline{Q}})$, because the Lang map is a finite \'etale morphism, a Schubert variety in the partial flag variety $\overline{G}/\bar \s({\overline{Q}})$. 

\subsection{Smoothness of Schubert varieties}
\label{subsec:smoothness-schubert}

It remains to analyze the smoothness of the Schubert varieties in partial flag varieties which arise in our situation. This type of question has been extensively studied. Let us list the few results that we will use below.

\begin{enumerate}
    \item All closures of Schubert varieties are normal. In particular, if
        its dimension is $\le 1$, then it is smooth. The consequence for us is that all strata of dimension $\le 1$, i.e., whenever $\ell(w)\le 1$, are smooth.
    \item
        If $S$ is itself a $\overline{Q}$-orbit (or more generally, a homogeneous space under any parabolic subgroup -- i.e., if the boundary in the sense of Section~\ref{subsec:cominuscule-parabolic} is empty), then $S$ is smooth. In the cases relevant to us, if $w'\in \overline{W}_{\bar \s(\overline{Q})}$, then $\overline{Q}w' \bar \s(\overline{Q})/\bar \s(\overline{Q}) = \overline{Q} \bar \s(\overline{Q})/\bar \s(\overline{Q})$ is closed in $\overline{G}/\bar \s(\overline{Q})$. This settles almost all cases where $\overline{Q}\ne \bar \s(\overline{Q})$.
    \item The case where $\overline{Q}$ is the maximal parabolic subgroup attached to a minuscule or cominuscule weight has been studied in~\cite{Lakshmibai-Weyman}, \cite{Brion-Polo}. See Section~\ref{subsec:cominuscule-parabolic} below for further details.
    \item Every Schubert variety in the full flag variety attached to a Coxeter element is smooth. For instance, this follows because in this case the Schubert variety is isomorphic to its Bott--Samelson resolution.
\end{enumerate}

There is another method to check the smoothness of Schubert varieties in the full flag variety of a  classical group, building on the pattern avoidance criteria in terms of the corresponding Weyl group element, viewed as a permutation, see~\cite{Billey-Lakshmibai} Ch.~8. Furthermore, every $\overline{Q}$-Schubert variety in $\overline{G}/\bar \s(\overline{Q})$ is smoothly equivalent to its inverse image in $\overline{G}/B$, the Schubert variety for the maximal element in $\overline{Q} w' \bar \s(\overline{Q})$. We will not resort to this method here.

\subsection{The case of a (co-) minuscule parabolic}
\label{subsec:cominuscule-parabolic}

Assume that $\overline{G}$ is absolutely simple. Since we will apply the following results in the situation that $\bar \s(\overline{Q})=\overline{Q}$ we change notation slightly here and consider Schubert varieties in $\overline{G}/\overline{Q}$.

We briefly state here the answers to the smoothness question in the case of a parabolic subgroup $\overline{Q}$ attached to a minuscule weight (in other words, the simple reflections in $\overline{Q}$ are precisely those stabilizing a fixed minuscule weight), and in the case of Dynkin type $C_n$ and $\overline{Q}$ the Siegel parabolic, i.e., the parabolic that belongs to the unique minuscule coweight (the ``cominuscule case''), cf.~\cite{Lakshmibai-Weyman}, \cite{Brion-Polo}.

We can start out slightly more generally than before and consider closures of $B$-orbits in $\overline{G}/\overline{Q}$, rather than only closures of $\overline{Q}$-orbits.  Given a Schubert variety $S = \overline{{\overline{B}}v\overline{Q}/\overline{Q}} \subseteq \overline{G}/\overline{Q}$, let $\overline{Q}' \subseteq \overline{G}$ be its stabilizer, and let the \emph{boundary} $\mathop{\rm Bd}(S)$ be the complement of the open $\overline{Q}'$-orbit in $S$ (cf.~\cite{Brion-Polo}). Note that $\overline{Q}'$ can be different from $B$ or $\overline{Q}$! Of course, if $S$ is the closure of a $\overline{Q}$-orbit, then $\overline{Q}\subseteq \overline{Q}'$, so if in addition $\overline{Q}$ is maximal and $S\ne \overline{G}/\overline{Q}$, then $\overline{Q}=\overline{Q}'$.

Note that in~\cite{GH}, end of proof of Prop.~7.3.2, we incorrectly claim that~\cite{Brion-Polo} Prop.~3.3 would imply that for all Schubert varieties $\overline{C}_{\overline{Q}, v}$ for $\overline{G}$ of type $B$ and $\overline{Q}$ the minuscule maximal parabolic, the singular locus equals the boundary in the above sense.

\subsubsection{The simply laced case}\label{An}

If $\overline{G}$ is simply laced, then by~\cite{Brion-Polo} Proposition 3.3 the boundary $\mathop{\rm Bd}(S)$ equals the singular locus of $S$.

\subsubsection{Type $B$, minuscule case}
\label{subsec:cominuscule-parabolic-bn}

Let ${\overline{G}}$ be of type $B_n$ (say ${\overline{G}}$ is the special orthogonal group $SO_{2n-1}$
over some algebraically closed field), fix a maximal torus and a Borel
subgroup, and let ${\overline{Q}}$ be the maximal parabolic subgroup attached to the unique
minuscule \emph{weight} $\omega_n$.

The result of~\cite[Section~5]{Lakshmibai-Weyman} in this case is the following.
We can identify ${}^{\overline{Q}}{\overline{W}}$ with the set
\[
    \{
        (d_1, \dots, d_n);\
        1 \le d_i \le 2n-1,\ \forall i,j: d_i + d_j \ne 2n+2
    \}
\]
(In other words, precisely one of $i$, $2n-i$ occurs, and $n+1$ must not occur.)

To an element $(d_i)_i$ we attach a partition $(a_1, \dots, a_n)$ by the rule
\[
    a_{n-i+2} = \left\{
        \begin{array}{ll}
            d_i - i & \text{if\ } d_i < n+1,\\
            d_i - i -1 & \text{if\ } d_i > n+1.
        \end{array}
    \right.
\]
Then $(a_i)_i$ is a self-dual partition fitting into an $n\times n$ square.

Now~\cite[Corollary~5.10]{Lakshmibai-Weyman} states that the Schubert variety
corresponding to $(d_i)_i$ is smooth if and only if the partition $(a_i)_i$ is
either a square, or a hook.

Among all the Schubert varieties, precisely the $n+1$ cases in the table below correspond to ${\overline{Q}}$-orbit closures:

\begin{table}[h]
\renewcommand{\arraystretch}{1.7}
\begin{tabular}{@{}l@{\hskip1cm}l@{\hskip1cm}l@{}}
    $(d_i)_i$ & $(a_i)_i$ & smooth/singular? \\\hline
    $(1, \dots, n)$ & $(0, \dots, 0)$ & smooth\\
    $(2, \dots, n, 2n+1)$ & $(n, 1 \dots, 1)$ & smooth\\
    $(3, \dots, n, 2n, 2n+1)$ & $(n, n, 2, \dots, 2)$ &  singular\\
    $(i, \dots, n, 2n+3-i, \dots, 2n-1)$ & $(n^{(i-1)}, (i-1)^{(n-i+1)})$ & singular  $(3 \le i \le n)$\\
    $(n+2, \dots, 2n+1)$ & $(n, \dots, n)$ & smooth\\
\end{tabular}
\end{table}

There is exactly one line where the Schubert variety is smooth but has non-empty
boundary, namely the ``hook'' in line 2 of the table.

\subsubsection{Type $C$, minuscule case}

Let ${\overline{G}}$ be of type $C_n$ (to pin things down, let us say ${\overline{G}}$ is the symplectic
group attached to a symplectic vector space $V$ of dimension $2n$ over some
algebraically closed field, $n\ge 2$), fix a maximal torus and a Borel subgroup,
and let ${\overline{Q}}$ be the maximal parabolic subgroup attached to the unique minuscule
\emph{weight} $\omega_1$; it is the stabilizer of a line $L\subset V$. Since ${\overline{G}}$
acts transitively on the set of all lines in $V$, we can identify ${\overline{G}}/{\overline{Q}}$ with the
projective space $\mathbb P(V)$. There are $2n$ different ${\overline{B}}$-orbits in ${\overline{G}}/{\overline{Q}}$,
and there closures are projective spaces of dimension $0,\dots, 2n-1
= \dim\mathbb P(V)$ --- the same Schubert varieties which arise from $GL(V)
\cong GL_{2n}$ and ${\overline{Q}}' \subset GL_{2n}$ the maximal parabolic ${\rm
Stab}_{GL(V)}(L)$.  In particular in this situation all Schubert varieties are
smooth.

Note though that the Schubert variety of dimension $d=2n-2$ (i.e., codimension
$1$) is the set of all lines in $L^\perp$ and is therefore fixed by $Q$. In
other words, it is equal to a ${\overline{Q}}$-orbit closure, and its stabilizer is equal to
${\overline{Q}}$.  It follows that the boundary in the above sense is not empty.

\subsubsection{Type $C$, cominuscule case}\label{Cn}

In this case, \cite{Brion-Polo} Prop.~4.4 shows that for all Schubert varieties
the boundary equals the singular locus.

\subsection{Results for the individual cases}

Now we restrict to the situation of Thm.~1.4, Table 1. Note that the
answer to our question will depend not only on the Coxeter datum, but on the
actual group over the finite field, i.e., on the orientation of the Dynkin
diagram.

The complement of $K$ in $\tS$ has either one or two elements. Consider a case where it has two elements, and exclude the $\tilde{A}_{n-1}\times \tilde{A}_{n-1}$ case.
It then turns out that the situation decomposes as a product, i.e., for
all $w\in \KCox(\mu)$ which do not lie in ${}^{K'}\Cox(\mu)$ for a larger $K'$,
$\supp_\s(w)$ is a disjoint union of two Dynkin diagrams, the group
$\overline{G}$ accordingly decomposes as a product, and the situation can be
analyzed by handling the two components separately.

\begin{theorem}
    Assume that $\bG$ is quasi-simple over $F$ and $\mu$ is non-central in every $\breve F$-simple component. Suppose that the triple $(\bG, \mu, K)$ is of Coxeter type. Then all closures of BT strata in $X(\mu, \t)_K$ are smooth, except for the following cases:
    \begin{enumerate}
        \item 
             the case $(\tilde A_{n-1}, \id, \o_1^\vee + \o_{n-1}^\vee, \tS-\{0\})$ for $n \ge 4$,
        \item
            $\dim X(\mu, \t)_K \ge 2$ and at least one of the long roots of the local Dynkin diagram is not contained in $K$.
 \end{enumerate}
\end{theorem}

\begin{remark}
    In this theorem, long roots and short roots occur in the local Dynkin diagrams whose associated affine Weyl groups are of type $\tilde B$ or $\tilde C$. In these cases, $\dim X(\mu, \t)_K<2$ only when the enhanced Coxeter datum is of type $(\tilde C_2, \id, \o_2^\vee, \{0\})$.
\end{remark}

\begin{proof}
    In the following cases, all strata have dimension $0$ or $1$:
    $(\tilde A_{n-1}, \id, \o_1^\vee, K)$,
    $(\tilde A_1, \id, 2 \o_1^\vee, K)$,
    $(\tilde A_3, \id, \o_2^\vee, K)$,
    $(\tilde C_2, \id, \o_2^\vee, K)$.
    It is clear that it is enough to check the case of minimal $K$.

    In the following cases, the corresponding Schubert variety has empty boundary (cf.~Section~\ref{subsec:smoothness-schubert}~(2)):
    $(\tilde A_{2m}, \varsigma_0, \o_1^\vee, K)$,
    $ (\tilde A_{2m+1}, \varsigma_0, \o_1^\vee, K)$,
    $(\tilde A_{n-1} \times \tilde A_{n-1}, {}^1 \varsigma_0, (\o_1^\vee, \o_{n-1}^\vee), K)$,
    $(\tilde A_3, \varsigma_0, \o_2^\vee, K)$,
    $(\tilde D_n, \id, \o_1^\vee, K)$,
    $(\tilde D_n, \varsigma_0, \o_1^\vee, K)$.
    It is enough to check the condition $w'\in \overline{W}_{\bar{\s}(\overline{Q})}$ for the minimal possible $K$, because for larger $K$, there are fewer elements $w'$ to check, and $\overline{Q}$ becomes larger.

    In the case $(\tilde A_{n-1}, \varrho_{n-1}, \o_1^\vee, \emptyset)$ the closures of strata are smoothly equivalent to Schubert varieties attached to Coxeter elements in a full flag variety, and hence smooth, see Section~\ref{subsec:smoothness-schubert}~(4).

    In the cases $(\tilde B_n, \Ad(\t_1), \o_1^\vee, \tS-\{n\})$ and $(\tilde C_2, \Ad(\t_2), \o_2^\vee, \{0, 2\})$, we have, depending on the orientation of the local Dynkin diagram, Schubert varieties in a group of type $B$ for a minuscule parabolic (leading to the smooth case in Line 2 of the table in Section~\ref{subsec:cominuscule-parabolic-bn}), or in a group of type $C_n$ for the cominuscule maximal parabolic (singular, unless of dimension $\le 1$).

    In the case $(\tilde B_n, \id, \o_1^\vee, \tS-\{0, n\})$ the situation decomposes as a product as explained above, where one factor contributes a smooth Schubert variety with empty boundary, and the other factor gives a Schubert variety of type $B$ or $C$, as in the previous paragraph. In the case $(\tilde C_n, \id, \o_1^\vee, \tS-\{0, n\})$ the situation decomposes as a product where both factors are of type $B$ or $C$, depending on the orientation of the local Dynkin diagram, as in the previous paragraph. If the level $K$ is a maximal proper subset of $\tS$, then only one of the two factors of this product decomposition occurs. The statement on the smoothness then follows from our discussion above.

    Regarding the case $(\tilde A_{n-1}, \id, \o_1^\vee + \o_{n-1}^\vee, \tS-\{0\})$, see Example~\ref{example-smoothness-1}.
\end{proof}

For Deligne--Lusztig varieties in orthogonal groups, cf.~also Oki's
paper~\cite{Oki}, Section~6. Note that this theorem corrects a few omissions of
smooth cases in~\cite{GH}~Section~7 and the erratum, and the incorrect claim
regarding non-smoothness in type $B$ in~\cite{GH}.

\begin{example}\label{example-smoothness-1}
    Type $(\tilde A_{n-1}, \id, \o_1^\vee + \o_{n-1}^\vee, \tS-\{0\})$. 
    
    If $n=3$, then $\KAdm(\mu)_0=\{1, s_0, s_0 s_1, s_0 s_2\}$. In this case, the closures of strata are smoothly equivalent to Schubert varieties attached to Coxeter elements in a full flag variety, and hence smooth, see Section~\ref{subsec:smoothness-schubert}~(4).
    
    If $n \ge 3$, then $s_0 s_1 s_{n-1} \in \KAdm(\mu)_0$. In this case, the Dynkin diagram of $\overline{G}$ is of type $A_3$ and consists of $s_0, s_1$ and $s_{n-1}$, and the parabolic subgroup $\overline{Q}$ is the standard parabolic subgroup of $\overline{G}$ associated to $\{s_1, s_{n-1}\}$. The Schubert variety $\overline{Q} s_0 s_1 s_{n-1} \overline{Q}/\overline{Q}$ has nonempty boundary, and hence is non-smooth by Section~\ref{An}. In fact, one may show by the same argument that the closure of a BT stratum of type $w \in \KAdm(\mu)_0$ is non-smooth if and only if $s_0 s_1 s_{n-1} \le w$. 
\end{example}

\begin{example}\label{example-smoothness-2}
    Type $(\tilde C_n, \id, \o_1^\vee, \tS-\{0, n\})$. 
    
    Let $-1 \le i \le j-2 \le n-1$ and $w=s_{[i, 0]} \i s_{[n, j]}$. Then $\supp(w)=\BS_1 \sqcup \BS_2$, where $\BS_1=\{0, 1, \cdots, i\}$ and $\BS_2=\{j, j+1 \cdots, n\}$. Set $w_1=s_{[i, 0]} \i$ and $w_2=s_{[n, j]}$. Let $\overline{G}_1$ and $\overline{G}_2$ be the connected semisimple groups associated to $\BS_1$ and $\BS_2$ respectively. Here if $i=-1$, then $\BS_1=\emptyset$ and $\overline{G}_1$ is the trivial group. Similarly, if $j=n+1$, then $\overline{G}_2$ is the trivial group. Let $\overline{B}_1$ and $\overline{B}_2$ be the standard Borel subgroups of $\overline{G}_1$ and $\overline{G}_2$ respectively. Let $\overline{Q}_1 \subset \overline{G}_1$ be the standard parabolic subgroup of type $\BS_1-\{0\}$ and $\overline{Q}_2 \subset \overline{G}_2$ be the standard parabolic subgroup of type $\BS_2-\{n\}$. Then the closure of a stratum of type $w$ is isomorphic to $X_1 \times X_2$, where $X_i=\{g \in \overline{G}_i/\overline{Q}_i; g \i \bar \s(g) \in \overline{\overline{Q}_i w_i \overline{Q}_i}\}$. 
    
    By Section~\ref{subsec:cominuscule-parabolic-bn} and Section~\ref{Cn}, $X_1$ is smooth if and only if $\{0\}$ is not the long root in the Dynkin diagram of $\BS_1$ or $i=0$ or $1$. Similarly, $X_2$ is smooth if and only if $\{n\}$ is not the long root in the Dynkin diagram of $\BS_2$ or $j=n$ or $n+1$. Hence the closure of all BT strata are smooth if and only if both $0$ and $n$ are short roots in the local Dynkin diagram of $\bG$. 
\end{example}

\end{document}